\documentclass[a4paper,12pt]{article}
\usepackage{amsthm}

\usepackage[T1]{fontenc} 
\usepackage[utf8]{inputenc}
\usepackage[ english]{babel}
\usepackage{amsmath}
\usepackage{amsfonts}
\usepackage{enumerate}
\usepackage{amssymb}
\usepackage[usenames]{color}
\usepackage{graphicx}
\usepackage{amsmath,dsfont}
\usepackage{hyperref}
\hypersetup{colorlinks=true, linkcolor=blue}

\newtheorem{theo}{Theorem}[section]
\newtheorem{lemm}[theo]{Lemma}

\newtheorem{prop}[theo]{Proposition}

\newtheorem{defi}[theo]{Definition}

\newtheorem{properties}[theo]{Properties}

\newcommand{\R}{\mathbb{R}}
\numberwithin{equation}{section}
\newcommand{\N}{\mathbb{N}}
\newcommand{\Z}{\mathbb{Z}}

\def\x{\textit {\textbf x}}

\def\lambdav{\boldsymbol{\lambda}}

\def\C{\mathcal{C}}

\def\u{\textit {\textbf u}}
\def\v{\textit {\textbf v}}

\def\<{\langle}
\def\>{\rangle}

\def\0{{\bf 0}}

\def\pup1{{\partial\u\over\partial x_1}}

\usepackage{graphicx}
\usepackage{graphics}
\usepackage{epsfig}
\usepackage{amsmath,fourier}
\usepackage{amsfonts}
\usepackage{psfrag}
\usepackage{subfigure,verbatim}

\usepackage[a4paper]{geometry}
\title{\bf{\textsl{Very weak solution for the exterior stationary Stokes equations with Navier slip boundary condition} }}

\author{Anis Dhifaoui \\}
\smallskip
\date{
{\small  UR Analysis and Control of PDEs, UR13ES64,\\
Department of Mathematics\\
Faculty of Sciences of Monastir\\
University of Monastir\\
5019 Monastir, Tunisia}\\
\bigskip
{\small Email adresses: anisdhifaoui123@gmail.com}}
\begin{document}

\maketitle

\begin{abstract}
In some problems of fluid mechanics, it is possible to be confronted with data that are not regular, that is why we are interested here in the search for the so-called very weak solutions for the stationary Stokes problem with Navier-type boundary conditions in a three-dimensional exterior domain. The problem describes the flow of a viscous and incompressible fluid past an obstacle where we assume that the fluid may slip on the boundary of the obstacle. Because the flow domain is unbounded, we set the problem in weighted Sobolev spaces in order to control the behavior at infinity of the solutions. Our purpose is to prove the existence and the uniqueness of a very weak solution in a Hilbertian framework.

%\noindent\textbf{AMS Subjets Classification}: 76D07, 35J25, 76D03, 35J50.
\end{abstract}
\section{Introduction}
We are interested in this paper on the existence of very weak solutions for the (linear stationary) Stokes problem in a domain $\Omega$ of $\R^3$. The set $\Omega$ is simply connected exterior domain, namely
the complement of a simply connected bounded domain which represents the obstacle. The system of stationary Stokes be written as follows:
\begin{eqnarray}\label{PS}
- \Delta\textbf{\textit{u}} +\nabla {\pi} = \textbf{\textit{f}} \quad \mathrm{and} \quad \mathrm{div}\, \textbf{\textit{u}}=0 \quad \mathrm{\text{ in }} \Omega,
\end{eqnarray}
where the unknowns are $\textit{\textbf{u}}$ the velocity of the fluid and its pressure $\pi$ and $\textbf{\textit{f}}$ the external forces acting on the fluid. To these equations, we supplement the
following Navier’s type slip boundary conditions:

\begin{equation}
\label{Navier.BC}
\textbf{\textit{u}}\cdot\textbf{\textit{n}}=0,\quad 2[\mathrm{\textbf{D}}(\textbf{\textit{u}})\textbf{\textit{n}}]_{\tau}+\alpha\textbf{\textit{u}}_{\tau}=\boldsymbol{0}
\quad\text{on}\quad\Gamma,
\end{equation}
where 
$\textbf{D}(\textbf{\textit{u}})=\dfrac{1}{2}\left(\nabla\,\textbf{\textit{u}}+\nabla\,\textbf{\textit{u}}^T\right)$ denotes the rate-of-strain tensor field, $\alpha$ is a scalar friction function, $\textbf{\textit{n}}$ is the unit normal vector to $\Gamma$
 and the notation $[\cdot]_\tau$ denotes the tangential component of a vector on $\Gamma$. The boundary condition~\eqref{Navier.BC}, proposed by H. Navier in 1827~\cite{NAVIER}. The first condition in~\eqref{Navier.BC} is the no-penetration condition and the second condition expresses the fact that the tangential 
velocity is proportional to the tangential stress. The Navier slip conditions have been extensively studied, see for instance~\cite{Achdou_CRAS_1995, Achdou_JCP_1998, Basson_CPAM_2008, 
Casado-Diaz_JDE_2003, Gerard_Varet_CMP_2010, Jager_JDE_2001, Beavers_JFM_1967, Solonnikov_TMIS_1973} and references therein.\\

\noindent The purpose of this paper is to study the exterior problem composed by the Stokes equations~\eqref{PS} and the Navier slip boundary conditions~\eqref{Navier.BC},
with a positive friction function $\alpha$ and where we also include the non homogeneous case.
Although the Stokes problem set in bounded domains with conditions~\eqref{Navier.BC} has been well studied by various authors 
(see for instance~\cite{Ahmed_2014, Beirao_ADE_2004, Amrita-2018} or~\cite{Amrouche_M2AS_2016, Solonnikov_TMIS_1973} for the case $\alpha=0$
and references therein), to the best of our knowledge, it is not the case when the domain is unbounded and $\alpha>0$. 
We can just mention~\cite{Russo_JDE_2011} where~\eqref{Navier.BC} was used for the stationary Navier-Stokes equations in exterior domains and \cite{DMR-2019} for the stationary  exterior Stokes equations, the authors in their articles they studied some results of existence and uniqueness for different types of solutions, including the variational solution $W^{1,2}_{0}(\Omega)\times L^{2}(\Omega)$ and the strong solutions  $W^{2,2}_{k+1}(\Omega)\times W^{1,2}_{k+1}(\Omega)$ for $k\in\Z$ .  For the case of Navier boundary conditions without friction ($\alpha=0$), let us mention~\cite{Meslamani_2013, LMR_2020}, where the following boundary conditions 
(also known as Hodge boundary conditions) were used:
\begin{equation}
 \label{Navier.flat.BC}
 \textbf{\textit{u}}\cdot\textbf{\textit{n}}=0,\quad\textbf{curl }\textit{\textbf{u}}\times\textbf{\textit{n}}=\boldsymbol{0}\quad\text{on}\quad\Gamma,
\end{equation}
where $\textbf{curl }\textit{\textbf{u}}$ is the vorticity field. These conditions coincide with~\eqref{Navier.BC} on flat boundaries when $\alpha=0$.
They were also used in~\cite{Beirao_CPAA_2006} for the study of the non stationary Navier-Stokes equations in half-spaces of $\R^3$.
We finally refer to~\cite{Mulone_Meccanica_1983, Mulone_AnnMatPuraAppl_1985}
for the study of the non stationary problem of Navier-Stokes with mixed boundary conditions that include~\eqref{Navier.BC} without friction.\\

\noindent The notion of very weak solution for the stationary Stokes or Navier-Stokes equations, corresponding to very irregular data, has been developed in the last years by Giga~\cite{Giga_81} and also by Lions and Magenes~\cite{Lions-Magenes_1968}  for the Laplace's equation. The Stokes problem when $\Omega$ is bounded with Dirichlet boundary condition has been studied by a large number of authors, from different points of view, here so we give some examples:  Giga and Sohr \cite{Giga},  Amrouche and Girault~\cite{Amrouche_CMJ_1994}  Galdi, Simader and Sohr~\cite{Galdi-2005},  K. Schumacher~\cite{Schumacher}, Amrouche and Rodriguez-Bellido~\cite{AMaria}, for the exterior domain we can mention Amrouche and Meslameni \cite{Meslameni-2013}. The very weak solutions of the problem~\eqref{PS}--~\eqref{Navier.BC} has been studied for a bounded domain in~\cite{Ahmed_2014} for $\alpha=0$, in~\cite{LMR_2020} for the exterior domain for the case $\alpha=0$ and in \cite{Yves} for the half-space.\\

\noindent The aim of this work is to study some results of existence, uniqueness of very weak solution for the stationary Stokes problem \eqref{PS} with the boundary condition \eqref{Navier.BC}. The study is based on a $L^2$-theory and, because the domain $\Omega$ is unbounded, we choose to set the problem in weighted spaces. The weight functions 
are polynomials and enable to describe the growth or the decay of functions at infinity which allows to 
look for solutions of~\eqref{PS}--\eqref{Navier.BC} with various behavior at infinity and this is one 
of the main advantages of the weighted spaces. One important question is to define rigorously the traces of the vector functions which are living in
subspaces of $W^{0,2}_{-k-1}(\Omega)$ (see Lemma~\ref{trace Navier p=2}).  We prove existence and uniqueness of very weak solutions $(\textbf{\textit{u}},\pi)$ belongs to $W^{0,2}_{-k-1}(\Omega)\times W^{-1,2}_{-k-1}(\Omega)$ (see Definition~\ref{définition formulation faible p=2}). The main idea here relies on the use of a duality argument using
the strong solutions obtained in~\cite{DMR-2019}.\\

\noindent \noindent The paper is organized as follows. In Section~\ref{sec.preliminaires}, we introduce the Notations,
the functional framework based on weighted Hilbert spaces. We recall the definitions of some spaces and their respective norms, besides some density results, characterization of dual space, we shall precise in which sense, the Navier slip boundary conditions are taken. 
We recall the main results on Stokes problem that we shall use. The main results of this paper are presented in Theorems~\ref{solutions très faibles p=2}~--~\ref{g neq 0} which proves the existence and uniqueness of very weak solution $(\textbf{\textit{u}},\pi)$ belongs to $W^{0,2}_{-k-1}(\Omega)\times W^{-1,2}_{-k-1}(\Omega)$, for $k\in\Z$. 
\section{Notations and preliminaries}
\label{sec.preliminaires}

%In this section, we first introduce the notation that will be used in the remaining of the paper. 
%Next, we define the functional framework that is based on weighted Hilbert spaces and we introduce properties related to the Navier boundary condition. 
%Finally, in order to make the paper as self-containt as possible, 
%we recall the existence and uniqueness results of the Stokes problem in the whole space $\R^3$.

\subsection{Notations}

\noindent Throughout this paper we assume that $\Omega'$ denotes a bounded open in $\R^3$ of class $\mathcal{C}^{2,1}$, simply connected bounded and with a connected boundary $\partial\Omega'=\Gamma$, representing an obstacle. Let $\Omega$ be the complement of $\overline{\Omega'}$ in $\R^3$, in other words an exterior domain. We will use bold characters for vector and matrix fields. Let $\N$ denote the set of non-negative integers and $\Z$ the set
of all integers . For any multi-index $\boldsymbol{\lambda}\in\N^3$, we denote
by $\partial^{\boldsymbol{\lambda}}$ the differential operator of order
$\boldsymbol{\lambda}$,
$$\partial^{\boldsymbol{\lambda}}=
{\partial^{|\lambdav|}\over\partial_1^{\lambda_1}\partial_2^{\lambda_2}\partial_3^{\lambda_3}},
\quad|\lambdav|=\lambda_1+\lambda_2+\lambda_2.$$
 \noindent For any $k\in\Z$, $\mathcal{P}_k$ stands for
the space of polynomials of degree less than or equal to $k$ and
$\mathcal{P}_k^{\Delta}$ the harmonic polynomials of $\mathcal{P}_k$. If $k$ is a
negative integer, we set by convention $\mathcal{P}_k=\{0\}$. We denote by $\mathcal{D}(\Omega)$ the space of
$\C^{\infty}$ functions with compact support in $\Omega$, $\mathcal{D}(\overline{\Omega})$ the restriction to $\Omega$ of functions belonging to 
$\mathcal{D}(\R^3)$. 
We recall that $\mathcal{D}'(\Omega)$ is the well-known space of distributions defined on
$\Omega$. We recall that $L^2(\Omega)$ is the well-known Lebesgue real space and for $m\ge1$, we recall that $H^{m}(\Omega)$ is the well-known Hilbert space $W^{m,2}(\Omega)$. We shall write $u\in
H_{loc}^{m}(\Omega)$ to mean that $u\in H^{m}(\mathcal{O})$, for any
bounded domain $\mathcal{O}$, with $\overline{\mathcal{O}}\subset\Omega$.
For any positive real number $R$, let $B_R$ denote the open ball centered at the origin, with radius $R$ and assuming that $R$ is sufficiently large for $\overline{\Omega'}\subset B_R$, we denote by $\Omega_R$ the intersection $\Omega \cap B_{R}$. The notation $\langle\cdot,\cdot\rangle$ will
denote adequate duality pairing and will be specified when needed. If not specified, $\langle\cdot,\cdot\rangle_{\Gamma}$ will denote the duality pairing between the space $H^{-1/2}(\Gamma)$ 
and its dual space $H^{1/2}(\Gamma)$. Given a Banach space $X$, with dual space $X'$ and a closed subspace $Y$ of $X$, we denote by $X'\perp Y$ the subspace of $X'$ orthogonal to $Y$, i.e.
\begin{equation*}
X'\perp Y=\big\lbrace f \in X'; <f,v>=0\quad \forall\, v \in X\big\rbrace=(X/Y)'.
\end{equation*}
The space $X'\perp Y$ is also called the polar space of $Y$ in $X'$. Given \textit{\textbf{A}} and \textit{\textbf{B}} two matrix fields, such that $\textit{\textbf{A}}=(a_{ij})_{1\leqslant i,j\leqslant 3}$ 
and $\textit{\textbf{B}}=(b_{ij})_{1\leqslant i,j\leqslant 3}$, then we define 
$\textit{\textbf{A}}:\textit{\textbf{B}}=(a_{ij}b_{ij})_{1\leqslant i,j\leqslant 3}$.
Finally, as usual, $C>0$ denotes a generic constant the value
of which may change from line to line and even at the same line.

\subsection{Weighted Hilbert spaces}

\noindent Let $\x=(x_1,x_2,x_3)$ be a typical point in $\R^3$ and let $r=|\x|=\left(x_1^2+x_2^2+x_3^2\right)^{1/2}$ denotes its distance to the origin. In order to control the behaviour at infinity of our functions and distributions we use for basic weights the quantity
$\rho(\x)=\left(1+r^2\right)^{1/2}$ which is equivalent to $r$ at infinity. For $k\in\Z$, we introduce
$$W_k^{0,2}(\Omega)=\Big\{u\in\mathcal{D}'(\Omega),\,\rho^k u\in L^2(\Omega)\Big\},$$
which is a Hilbert space equipped with the norm: 
$$\|u\|_{W_k^{0,2}(\Omega)}=\|\rho^k u\|_{L^2(\Omega)}.$$ 
Let $m\geqslant1$ be an integer. We define the weighted Hilbert space:
$$
W_{k}^{m,2}(\Omega)=\Big\{u\in \mathcal{D}'(\Omega);\,\forall\boldsymbol{\lambda}\in\N^{3}:
\,0\leq |\boldsymbol{\lambda}| \leq m,\,\rho^{k-m+|\boldsymbol{\lambda}|}\partial^{\boldsymbol{\lambda}}u \in L^{2}(\Omega) \Big\},
$$
equipped with the norm
$$
\|u\|_{W_{k}^{m,2}(\Omega)}= \left(\sum_{0\leqslant|\boldsymbol{\lambda}|\leqslant m}
\|\rho^{k-m+|\boldsymbol{\lambda}|}\partial^{\boldsymbol{\lambda}}u\|^{2}_{L^{2}(\Omega)}\right)^{1/2}.
$$
\noindent We define the semi-norm
$$|u|_{W_k^{m,2}(\Omega)}=\left(\sum_{|\boldsymbol{\lambda}|=m}\|\rho^k\partial^{\boldsymbol{\lambda}}u\|_{L^2(\Omega)}\right)^{1/2}.$$
\noindent Let us give an examples of such spaces that will be often used in the remaining of the paper. 
For $m=1$, we have
$$W_{k}^{1,2}(\Omega)=\Big\{u\in \mathcal{D}'(\Omega),\,\rho^{k-1}u\in L^{2}(\Omega),\,\rho^{k}\nabla u\in {L}^{2}(\Omega)\Big\}.$$
For $m=2$, we have
$$
W_{k+1}^{2,2}(\Omega):=\left\lbrace u\in {W_{k}^{1,2}(\Omega)}, \rho^{k+1}\nabla^{2} u \in L^{2}(\Omega)\right\rbrace. 
$$
\noindent For the sake of simplicity, we have defined these spaces with integer exponents on the weight function. 
But naturally, these definitions can be extended to real number exponents with eventually some slight modifications~(see\cite{Amrouche_1994} for more details).\\

\noindent We shall now give some basic properties of those spaces:
\begin{properties}
\label{prop.proprietes.espaces.poids}
\
\begin{enumerate}
\item The space $\mathcal{D}(\overline{\Omega})$ is dense
in $W_k^{m,2}(\Omega).$
\item For any $\lambdav\in\N^3$, the mapping
\begin{equation}
\label{derive.espaces.poids}
u\in W_{k}^{m,2}(\Omega)\rightarrow\partial^{\boldsymbol{\lambda}}u\in W_{k}^{m-|\boldsymbol{\lambda}|,2}(\Omega)
\end{equation}
is continuous.
\item For $m\in\N\setminus\{0\}$ and $k\in\Z$,  we have the following continuous imbedding:
\begin{equation}\label{inclusion.sobolev}
W_{k}^{m,2}(\Omega)\hookrightarrow W_{k-1}^{m-1,2}(\Omega).
\end{equation}
\item The space $\mathcal{P}_{m-2-k}$ is the space of all polynomials included in $W_k^{m,2}(\Omega)$ and the following Poincar\'e-type inequality holds: 
\begin{equation}\label{thhardy}
\forall u\in W_k^{m,2}(\Omega),\quad\inf_{\mu\in\mathcal{P}_{j'}}\|u+\mu\|_{W_k^{m,2}(\Omega)}\leqslant
C|u|_{W_k^{m,2}(\Omega)},
\end{equation}
where $j'=\min(m-2-k,m-1)$. In other words the semi-norm $|\cdot|_{W_k^{m,2}(\Omega)}$ is a norm on $W_k^{m,2}(\Omega)/\mathcal{P}_{j'}$.
In particular $|\cdot|_{W_0^{1,2}(\Omega)}$ is a norm on $W_0^{1,2}(\Omega)$.
\end{enumerate}
\end{properties}
\noindent For more details on the above properties and the ones that we shall present in the following, 
the reader can refer to~\cite{Hanouzet_1971, KUFNER, Amrouche_1994, Amrouche_JMPA_1997} and references therein.\\ 

\noindent Note that all the local properties of the space $W_k^{m,2}(\Omega)$ coincide with those of the
standard Hilbert spaces $H^{m,2}(\Omega)$. Hence, it also satisfies the
usual trace theorems on the boundary $\Gamma$. Therefore, we can
define the space
$$\mathring{W}_k^{m,2}(\Omega)=\Big\{u\in W_k^{m,2}(\Omega),\,\gamma_0u=0,\,\gamma_1u=0,\cdots,\gamma_{m-1}u=0\text{ on }\Gamma\Big\}.$$
If $\Omega$ is the whole space $\R^3$, then the spaces 
$\mathring{W}_k^{m,2}(\R^3)$ and $W_{k}^{m,2}(\R^3)$ coincide. The space
$\mathcal{D}(\Omega)$ is dense in $\mathring{W}_k^{m,2}(\Omega)$.
Therefore, the dual space of $\mathring{W}_k^{m,2}(\Omega)$,
denoted by $W_ {-k}^{-m,2}(\Omega)$ is a space of
distributions. For $m\in\N$ and $k\in\Z$, we have the continuous imbedding 
\begin{equation}\label{inclusion.sobolev2}
W_k^{-m,2}(\Omega)\subset W_{k-1}^{-m-1,2}(\Omega).
\end{equation}
Moreover, we have the following Poincar\'e-type
inequality:
\begin{equation}\label{thHardy2}
\forall u\in\mathring{W}_k^{m,2}(\Omega),\quad\|u\|_{W_k^{m,2}(\Omega)}\le C|u|_{W_k^{m,2}(\Omega)}.
\end{equation}

\noindent We now introduce some weighted Hilbert spaces that are specific for the study of the Stokes problem~\eqref{PS} with the Navier boundary conditions \eqref{Navier.BC}. 
We start by introducing the following space for $k \in \Z$:
\begin{equation*}
H_{k}(\mathrm{div};\Omega)=\Big\{ \textbf{\textit{v}}\in W_{k}^{0,2}(\Omega),\,\mathrm{div}\,\,\textbf{\textit{v}}\in W_{k+1}^{0,2}(\Omega)\Big\}.
\end{equation*} 
\noindent This space is endowed with the norm

$$ 
\|\textbf{\textit{v}}\|_{H_{k}(\mathrm{div};\Omega)}=\left( \|\textbf{\textit{v}}\|^{2}_{W_{k}^{0,2}(\Omega)}+\|\mathrm{div} \,\, \textbf{\textit{v}}\|^{2}_{{W}_{k+1}^{0,2}(\Omega)}\right)^{1/2}.
 $$
\noindent Observe that $\mathcal{D}(\overline{\Omega})$ is dense in $H_{k}(\mathrm{div};\Omega)$.  
For the proof, one can use the same arguments as for the proof of the density of $\mathcal{D}(\overline{\Omega})$ in $W_{k}^{m,2}(\Omega)$ (see \cite{Hanouzet_1971}). 
Therefore, denoting by $\textbf{\textit{n}}$ the unit normal vector to the boundary $\Gamma$ pointing outside $\Omega$, if $\textit{\textbf{v}}$ belongs to $H_{k}(\mathrm{div};\Omega)$, 
then $\textit{\textbf{v}}$ has a normal trace $\textbf{\textit{v}}\cdot \textbf{\textit{n}}$ in $H^{-1/2}(\Gamma)$ and  there exists $C>0$
such that
\begin{equation}
\label{theorem.trace.normale}
 \forall \textbf{\textit{v}} \in H_{k}(\mathrm{div};\Omega),\quad\|\textbf{\textit{v}}\cdot \textbf{\textit{n}}\,\|_{ H^{-1/2}(\Gamma)}\leq C
 \|\textbf{\textit{v}}\,\|_{H_{k}(\mathrm{div};\Omega)}.
\end{equation}

\noindent Moreover the below Green's formula holds. For any $\textbf{\textit{v}} \in H_{k}(\mathrm{div};\Omega)$ and $\boldsymbol{\varphi} \in W_{-k}^{1,2}(\Omega)$, 
we have 
\begin{equation}\label{FG1}
\left\langle \textbf{\textit{v}}\cdot\textbf{\textit{n}}, \boldsymbol{\varphi}\right\rangle_{\Gamma}
=\int_{\Omega}\textbf{\textit{v}}\cdot\nabla\,\boldsymbol{\varphi}\,d\textbf{\textit{x}}+\int_{\Omega}\boldsymbol{\varphi}\,\mathrm{div}\,\textbf{\textit{v}}\,d\textbf{\textit{x}}.
\end{equation}
The closure of $\mathcal{D}({\Omega})$ in ${H}_{k}(\mathrm{div};\Omega)$ is denoted by $\mathring{H}_{k}(\mathrm{div};\Omega)$ and can be characterized by
$$\mathring{H}_{k}(\mathrm{div};\Omega)=\Big\{ \textbf{\textit{v}}\in {H}_{k}(\mathrm{div};\Omega),\,
\textbf{\textit{v}}\cdot\textbf{\textit{n}}=0\,\text{ on }\, \Gamma\Big\}.$$
Its dual space is denoted by ${H}_{-k}^{-1}(\mathrm{div};\Omega)$  and is characterized by the below proposition.

\begin{prop}
\label{caracterisation du H'}
Assume that $\Omega$ is of class $\mathcal{C}^{1,1}$ and let $k\in\Z$.
A distribution  $\textbf{\textit{f}}$ belongs to ${H}_{k}^{-1}(\mathrm{div};\Omega)$ if and only if there exist $\boldsymbol{\psi}\in W^{0,2}_{k}(\Omega)$ and $\chi\in W^{0,2}_{k-1}(\Omega)$, 
such that $\textbf{\textit{f}}=\boldsymbol{\psi}+\nabla\chi$. Moreover
$$
\|\boldsymbol{\psi}\|_{W^{0,2}_{k}(\Omega)} + \|\chi\|_{W^{0,2}_{k-1}(\Omega)}\leq C\|\textbf{\textit{f }}\|_{{H}_{k}^{-1}(\mathrm{div};\Omega)}.
$$
 \end{prop}
\noindent The proof of Proposition~\ref{caracterisation du H'} can be found in~\cite[Proposition 1.3]{Meslamani_2013}.
A consequence of this proposition and the imbedding~\eqref{inclusion.sobolev2} is that, for any $k\in\Z$, we have the imbedding
\begin{equation}
\label{inclusion.Hdiv}
H_k^{-1}(\mathrm{div};\Omega)\subset W_{k-1}^{-1,2}(\Omega).
\end{equation} 
\subsection{Density and trace results}
In this work, we are going to study the existence and uniqueness of very weak solutions for the stationary Stokes problem~\eqref{PS}--\eqref{Navier.BC}. Let us recall some notations related to the boundary condition. First, for any vector field $\textbf{\textit{v}}$ on $\Gamma$, we can write
\begin{equation}
\label{decomposition.trace}
\textbf{\textit{v}}=\textbf{\textit{v}}_{\tau}+(\textbf{\textit{v}}\cdot\textbf{\textit{n}})\textbf{\textit{n}},
\end{equation}
where $\textbf{\textit{v}}_{\tau}$ is the projection of $\textbf{\textit{v}}$ on the tangent hyper-plan to $\Gamma$. Next, for any  point  $\textbf{\textit{x}}$ on $\Gamma$, one may choose an open neighbourhood $\mathcal{W}$ of $\textbf{\textit{x}}$ in $\Gamma$ small enough to allow the existence of two families of $\mathcal{C}^2$ curves on $\mathcal{W}$ and where the lengths $s_{1}$ and $s_{2}$ along each family of curves are possible system of coordinates. Denoting by $\boldsymbol{\tau}_1$, $\boldsymbol{\tau}_2$ the unit tangent vectors to each family of curves, we have
 $$\textbf{\textit{v}}_{\tau}=(\textbf{\textit{v}}\cdot\boldsymbol{\tau}_{1})\boldsymbol{\tau}_{1}+(\textbf{\textit{v}}\cdot\boldsymbol{\tau}_{2})\boldsymbol{\tau}_{2}.
$$
As a result for any $\textbf{\textit{v}}\in \mathcal{D}(\overline{\Omega})$ the following formula holds (see \cite[Lemma 2.1]{Ahmed_2014})
\begin{eqnarray}\label{tenseur de déformation sur le bord}
2[\mathrm{\textbf{D}}(\textbf{\textit{v}})\textbf{\textit{n}}]_{\tau}=\nabla_{\tau}(\textbf{\textit{v}}\cdot\textbf{\textit{n}})+\left( \frac{\partial \textbf{\textit{v}}}{\partial \textbf{\textit{n}}}\right)_{\tau}-\Lambda \textbf{\textit{v}}\qquad\text{ on }\quad\Gamma,
\end{eqnarray}
 where
\begin{eqnarray*}
\Lambda\,\textbf{\textit{v}}=\sum_{k=1}^{2}\left(\textbf{\textit{v}}_{\tau}\cdot\frac{\partial\textbf{\textit{n}}}{\partial s_{k}} \right) \boldsymbol{\tau}_{k}.
\end{eqnarray*}
\noindent Now, we need to introduce some functional spaces. Otherwise, it is necessary to establish some Green formulas which are deduce from density lemmas. We start by the following space:
\begin{eqnarray*}
T_{k+1}^{}(\Omega)=\Big\lbrace\textbf{\textit{v}}\in \mathring{H}_{k-1}^{}(\mathrm{div},\Omega);\,\,\,\, \mathrm{div}\,\textbf{\textit{v}}\in \mathring{W}^{1,2}_{k+1}(\Omega)\Big\rbrace,
\end{eqnarray*}
which is a Hilbert space equipped with the following norm
\begin{eqnarray*}
\Vert\textbf{\textit{v}}\Vert_{T_{k+1}(\Omega)}=\Vert\textbf{\textit{v}}\Vert_{{W}^{0,2}_{k-1}(\Omega)}+\Vert\mathrm{div}\,\textbf{\textit{v}}\Vert_{{W}^{1,2}_{k+1}(\Omega)}
\end{eqnarray*}
The following Lemma state some properties related to the space $T_{k+1}(\Omega)$.
\begin{lemm}\label{caractérisation T^2} Let $k\in \Z$, the following properties hold:\quad\\
\begin{itemize}
\item[1-] A distribution $\textbf{\textit{f}}$ belongs to $(T_{k+1}(\Omega))'$ if and only if there exist $\boldsymbol{\phi}\in W^{0,2}_{-k+1}(\Omega)$ and $f_0\in W^{-1,2}_{-k-1}(\Omega)$, such that 
\begin{eqnarray*}\label{caractérisation T}
\textbf{\textit{f}}=\boldsymbol{\phi}+\nabla\,f_0
\end{eqnarray*}
\item[2-] The space $\mathcal{D}(\Omega)$ is dense in $T_{k+1}(\Omega)$. \\
 \item[3-] For any $\chi\in W^{-1,2}_{-k-1}(\Omega)$ and $\textbf{\textit{v}}\in T_{k+1}(\Omega)$, we have
\begin{eqnarray}\label{formule d'integration dans T^2}
\langle\nabla\,\chi,\textbf{\textit{v}}\rangle_{(T_{k+1}(\Omega))'\times T_{k+1}(\Omega)}=-\langle\chi,\mathrm{div}\,\textbf{\textit{v}}\rangle_{W^{-1,2}_{-k-1}(\Omega)\times \mathring{W}^{1,2}_{k+1}(\Omega)}.
\end{eqnarray}
\end{itemize}
\end{lemm}
\noindent Giving a meaning to Navier boundary conditions of a very weak solution of a Stokes problem is not easy. For this reason, we need to introduce the following Hilbert space, for $k\in\Z$:
\begin{eqnarray*}
{H}^{}_{-k-1}(\Delta;\Omega)=\Big\lbrace \textbf{\textit{v}}\in W^{0,2}_{-k-1}(\Omega);\,\,\, \Delta\,\textbf{\textit{v}}\in (T^{}_{k+1}(\Omega))'\Big\rbrace,
\end{eqnarray*}

which has the following norm:
\begin{eqnarray*}
\Vert\textbf{\textit{v}}\Vert_{{H}_{-k-1}(\Delta;\Omega)}=\Vert\textbf{\textit{v}}\Vert_{W^{0,2}_{-k-1}(\Omega)}+\Vert\Delta\,\textbf{\textit{v}}\Vert_{(T^{}_{k+1}(\Omega))'}
\end{eqnarray*}
%The following lemma will help us to prove a trace result.
%The next lemma will help us to prove a trace result.
\begin{lemm}\label{desnité dans HDelta}
The space $\mathcal{D}(\overline{\Omega})$ is dense in ${H}_{-k-1}(\Delta;\Omega)$.
\end{lemm}
\noindent The proofs of Lemma~\ref{caractérisation T^2} and Lemma~\ref{desnité dans HDelta} can be found in \cite{LMR_2020}. Finally, in order to write a Green formula, we define for $k\in\Z$:
\begin{eqnarray*}
{S}^{}_{k+1}(\Omega)=\Big\lbrace\textbf{\textit{v}}\in W^{2,2}_{k+1}(\Omega);\quad \mathrm{div}\,\textbf{\textit{v}}=0,\,\,\ \textbf{\textit{v}}\cdot\textbf{\textit{n}}=0,\,\, 2[\textbf{D}(\textbf{\textit{v}})\textbf{\textit{n}}]_{\tau}+\alpha\textbf{\textit{v}}_{\tau}=0\,\,\text{on}\,\,\Gamma \Big\rbrace
\end{eqnarray*}
The next lemma will help us to prove a trace result.
%The following lemma proves that for any $\textbf{\textit{v}}$ which belongs to $\textbf{H}(\Delta;\Omega)$, we have $[\textbf{D}(\textbf{\textit{v}})\textbf{\textit{n}}]_{\tau}$ is we defined in $H^{-3/2}(\Gamma)$. 
\begin{lemm}\label{trace Navier p=2}
 Let $\alpha$ belongs to $W^{1,\infty}(\Gamma)$. The linear mapping $\Theta$ : $\textbf{\textit{u}}\mapsto 2[\textbf{D}(\textbf{\textit{u}})\textbf{\textit{n}}]_{\tau}+\alpha\textbf{\textit{u}}_{\tau}$ defined on $\mathcal{D}(\overline{\Omega})$ can be extended to a linear and continuous mapping still denoted by $\Theta$, from ${H}_{-k-1}(\Delta;\Omega)$ into $H^{-3/2}(\Gamma)$ and we have the following Green's formula: for any $\textbf{\textit{u}}\in {H}_{-k-1}(\Delta;\Omega)$ and $\boldsymbol{\varphi}\in {S}_{k+1}(\Omega)$,
\begin{eqnarray}\label{Green formula very weak p=2}
\nonumber\langle\Delta\,\textbf{\textit{u}},\boldsymbol{\varphi}\rangle_{(T^{2}_{k+1}(\Omega))'\times T^{2}_{k+1}(\Omega)}&=&\displaystyle\int_{\Omega} \textbf{\textit{u}}\cdot \Delta\,\boldsymbol{\varphi} d\textbf{x}+\langle 2[\textbf{D}(\textbf{\textit{u}})\textbf{\textit{n}}]_{\tau}+\alpha\textbf{\textit{u}}_{\tau},\boldsymbol{\varphi}\rangle_{H^{-3/2}(\Gamma)\times H^{3/2}(\Gamma)}\\
 && -\langle 2[\textbf{D}(\boldsymbol{\varphi})\textbf{\textit{n}}]\cdot\textbf{\textit{n}},\textbf{\textit{u}}\cdot\textbf{\textit{n}}\rangle_{H^{1/2}(\Gamma)\times H^{-1/2}(\Gamma)}.
\end{eqnarray}
\end{lemm}
%The proof of this Lemma is given later in Lemma~\ref{trace Navier bien déf} (cas générale dans $L^p$).\\
\begin{proof}
 The goal is to prove that the mapping~$\Theta$ defined on 
$\mathcal{D}(\overline{\Omega})$ is continuous for the norm of ${H}^{}_{-k-1}(\Delta ;\Omega)$. Let $\textbf{\textit{u}}\in\mathcal{D}(\overline{\Omega})$ and $\boldsymbol{\varphi}\in {S}^{}_{k+1}(\Omega)$. Then thanks to the following identity
\begin{equation}\label{delta2}
\Delta\,\textbf{\textit{u}}=2\mathrm{div}\,\mathrm{\textbf{D}}(\textbf{\textit{u}})-\nabla\,\mathrm{div}\,\textbf{\textit{u}}
\end{equation}
\noindent and the Green formula \eqref{FG1}, we have 
\begin{eqnarray*}\label{calcule delta phi}
\nonumber\left\langle\Delta\,\textbf{\textit{u}},\boldsymbol{\varphi} \right\rangle _{(T^{2}_{k+1}(\Omega))'\times T^{2}_{k+1}(\Omega))}=-2\displaystyle\int_{\Omega}\mathrm{\textbf{D}}(\textbf{\textit{u}}):
\mathbf{\nabla}\boldsymbol{\varphi}\,d\textbf{\textit{x}}
+\displaystyle\int_{\Omega}\mathrm{div}\,\textbf{\textit{u}}\,\mathrm{div}\,\boldsymbol{\varphi} d\textbf{\textit{x}}+2\langle\mathrm{\textbf{D}}(\textbf{\textit{u}})\textbf{\textit{n}},\boldsymbol{\varphi}\rangle_{\Gamma}\\
%&=&\displaystyle\int_{\Omega}\textbf{\textit{u}}\cdot\Delta\,\boldsymbol{\varphi} d\textbf{x}+2\langle\mathrm{\textbf{D}}(\textbf{\textit{u}})\textbf{\textit{n}},\boldsymbol{\varphi}\rangle_{\Gamma}-2\langle\textbf{D}(\boldsymbol{\varphi})\textbf{\textit{n}},\textbf{\textit{u}}\rangle_{\Gamma}.
\end{eqnarray*}

\noindent Next, the fact that $\boldsymbol{\varphi}\cdot\textbf{\textit{n}}=0$ on $\Gamma$ also implies that
\begin{equation*}
\label{tenseur de deformation}
\langle\mathrm{\textbf{D}}(\textbf{\textit{u}})\textbf{\textit{n}},\boldsymbol{\varphi}\rangle_{\Gamma} =
\left\langle\big(\big[\textbf{D}(\textbf{\textit{u}})\textbf{\textit{n}}\big]\cdot\textbf{\textit{n}}\big)\textbf{\textit{n}}+ 
[\textbf{D}(\textbf{\textit{u}})\textbf{\textit{n}}]_{\tau},\boldsymbol{\varphi}\right\rangle_\Gamma
=\langle\mathrm{[\textbf{D}}(\textbf{\textit{u}})\textbf{\textit{n}}]_{\tau},\boldsymbol{\varphi}\rangle_{\Gamma}.
\end{equation*}

\noindent It follows that, for any $\textbf{\textit{u}}\in \mathcal{D}(\overline{\Omega})$ and  $\boldsymbol{\varphi}\in {S}_{k+1}(\Omega)$, we have

$$-\left\langle\Delta\,\textbf{\textit{u}},\boldsymbol{\varphi} \right\rangle _{\Omega}=2\displaystyle\int_{\Omega}\mathrm{\textbf{D}}(\textbf{\textit{u}})
:\nabla\boldsymbol{\varphi}\,d\textbf{\textit{x}}-2\left\langle[\mathrm{\textbf{D}}(\textbf{\textit{u}})\textbf{\textit{n}}]_{\tau},\boldsymbol{\varphi} \right\rangle_{\Gamma}.$$

\noindent Since we have 

$$\int_\Omega\textbf{D}(\textit{\textbf{u}}):\nabla\boldsymbol{\varphi}\,d\textit{\textbf{x}}=
\int_\Omega\textbf{D}(\textit{\textbf{u}}):\textbf{D}(\boldsymbol{\varphi})\,d\textit{\textbf{x}},$$
Then, we obtain
\begin{eqnarray}\label{calcule delta phi}
\left\langle\Delta\,\textbf{\textit{u}},\boldsymbol{\varphi} \right\rangle _{(T^{2}_{k+1}(\Omega))'\times T^{2}_{k+1}(\Omega))}=\displaystyle\int_{\Omega}\textbf{\textit{u}}\cdot\Delta\,\boldsymbol{\varphi} d\textbf{x}+\left\langle 2[\mathrm{\textbf{D}}(\textbf{\textit{u}})\textbf{\textit{n}}]_{\tau},\boldsymbol{\varphi} \right\rangle_{\Gamma}-\left\langle 2\mathrm{\textbf{D}}(\boldsymbol{\varphi})\textbf{\textit{n}},\textbf{\textit{u}} \right\rangle_{\Gamma}.
\end{eqnarray}
%\noindent It follows that, for any $\textbf{\textit{u}}\in \mathcal{D}(\overline{\Omega})$ and  $\boldsymbol{\varphi}\in V_{\sigma,T}(\Omega)$, we have
\noindent Now, since $2[\textbf{D}(\boldsymbol{\varphi})\textbf{\textit{n}}]_{\tau}=-\alpha\boldsymbol{\varphi}_{\tau}$ on $\Gamma$, we have
\begin{eqnarray*}
-2\langle\textbf{D}(\boldsymbol{\varphi})\textbf{\textit{n}},\textbf{\textit{u}}\rangle=\langle\alpha\boldsymbol{\varphi}_{\tau},\textbf{\textit{u}}_{\tau}\rangle_{\Gamma}-2\langle[\textbf{D}(\boldsymbol{\varphi})\textbf{\textit{n}}]\cdot\textbf{\textit{n}},\textbf{\textit{u}}\cdot\textbf{\textit{n}} \rangle_{\Gamma}
\end{eqnarray*}
Plunging this in~\eqref{calcule delta phi}, for any $\textbf{\textit{u}}\in \mathcal{D}(\overline{\Omega})$ and $\boldsymbol{\varphi}\in {S}_{k+1}(\Omega)$ we have
\begin{eqnarray*}
\label{formule de green}
\left\langle\Delta\,\textbf{\textit{u}},\boldsymbol{\varphi} \right\rangle _{(T^{2}_{k+1}(\Omega))'\times T^{2}_{k+1}(\Omega))}&=&\displaystyle\int_{\Omega}\textbf{\textit{u}}\cdot\Delta\,\boldsymbol{\varphi} d\textbf{x}+\left\langle 2[\mathrm{\textbf{D}}(\textbf{\textit{u}})\textbf{\textit{n}}]_{\tau}+\alpha\textbf{\textit{u}}_{\tau},\boldsymbol{\varphi} \right\rangle_{H^{-3/2}(\Gamma)\times H^{3/2}(\Omega)}\\
&&-2\langle [\textbf{D}(\boldsymbol{\varphi})\textbf{\textit{n}}]\cdot\textbf{\textit{n}},\textbf{\textit{u}}\cdot\textbf{\textit{n}}\rangle_{H^{1/2}(\Gamma)\times H^{-1/2}(\Gamma)}.
\end{eqnarray*}

%\noindent Let us now introduce some notations related to the boundary. First,  Next, for any  point  $\textbf{\textit{x}}$ on $\Gamma$, one may choose an open neighbourhood $\mathcal{W}$ of $\textbf{\textit{x}}$ in $\Gamma$ small enough to allow the existence of two families of $\mathcal{C}^2$ curves on $\mathcal{W}$ and where the lengths $s_{1}$ and $s_{2}$ along each family of curves are possible system of coordinates. Denoting by $\boldsymbol{\tau}_1$, $\boldsymbol{\tau}_2$ the unit tangent vectors to each family of curves, we have
% $$\textbf{\textit{v}}_{\tau}=(\textbf{\textit{v}}\cdot\boldsymbol{\tau}_{1})\boldsymbol{\tau}_{1}+(\textbf{\textit{v}}\cdot\boldsymbol{\tau}_{2})\boldsymbol{\tau}_{2}.
%$$
%As a result for any $\textbf{\textit{v}}\in \mathcal{D}(\overline{\Omega})$ the following formula holds (see \cite[Lemma 2.1]{Ahmed_2014})
%\begin{eqnarray}\label{tenseur de déformation sur le bord}
%2[\mathrm{\textbf{D}}(\textbf{\textit{v}})\textbf{\textit{n}}]_{\tau}=\nabla_{\tau}(\textbf{\textit{v}}\cdot\textbf{\textit{n}})+\left( \frac{\partial \textbf{\textit{v}}}{\partial \textbf{\textit{n}}}\right)_{\tau}-\Lambda \textbf{\textit{v}}\qquad\text{ on }\quad\Gamma,
%\end{eqnarray}
% where
%\begin{eqnarray*}
%\Lambda\,\textbf{\textit{v}}=\sum_{k=1}^{2}\left(\textbf{\textit{v}}_{\tau}\cdot\frac{\partial\textbf{\textit{n}}}{\partial s_{k}} \right) \boldsymbol{\tau}_{k}.
%\end{eqnarray*}
\noindent Let now $\boldsymbol{\mu}$ be any element of $H^{3/2}(\Gamma)$. Since $\alpha\in W^{1,\infty}(\Gamma)$, then $\alpha\boldsymbol{\mu}_{\tau}\in H^{1/2}(\Gamma)$. So there exists an element $\boldsymbol{\varphi}$ in $W^{2,2}_{k+1}(\Omega)$ such that
$$ \boldsymbol{\varphi}=\boldsymbol{\mu}_{\tau}\quad\text{ and } \quad \dfrac{\partial\boldsymbol{\varphi}}{\partial\textbf{\textit{n}}}=\Lambda\,\boldsymbol{\mu}-\textbf{\textit{n}}\,\mathrm{div}\,_{\Gamma}\boldsymbol{\mu}_{\tau}-\alpha\boldsymbol{\mu}_{\tau}\quad\text{ on }\,\, \Gamma. $$
Where \begin{eqnarray*}
\Lambda\,\boldsymbol{\mu}=\sum_{k=1}^{2}\left(\boldsymbol{\mu}_{\tau}\cdot\frac{\partial\textbf{\textit{n}}}{\partial s_{k}} \right) \boldsymbol{\tau}_{k}.
\end{eqnarray*}
As $\Lambda\,\boldsymbol{\mu}\cdot\textbf{\textit{n}}=0$ on $\Gamma$, we have 
\begin{eqnarray*}
\dfrac{\partial\boldsymbol{\varphi}}{\partial\textbf{\textit{n}}}\cdot\textbf{\textit{n}}=-\mathrm{div}_{\Gamma}\,\boldsymbol{\mu}_{\tau}\quad\text{ on}\,\,\Gamma
\end{eqnarray*}
and we recall the following formula (see \cite{Amrouche_CMJ_1994}):
\begin{eqnarray}\label{div sur trace p=2}
\mathrm{div}\textbf{\textit{v}}=\mathrm{div}_{\Gamma}\,\textbf{\textit{v}}_{\tau}+\boldsymbol{\beta}\,\textbf{\textit{v}}\cdot\textbf{\textit{n}}+\dfrac{\partial\,\textbf{\textit{v}}}{\partial\,\textbf{\textit{n}}}\cdot\textbf{\textit{n}}\quad\text{ on }\,\,\Gamma,
\end{eqnarray}
where $\boldsymbol{\beta}$ denotes the mean curvature of $\Gamma$, $\mathrm{div}_{\Gamma}$ is the surface divergence.
 We deduce that
\begin{eqnarray*}
\mathrm{div}\,\boldsymbol{\varphi}=0\quad\text{ on }\,\,\Gamma.
\end{eqnarray*}
 Using relations~\eqref{tenseur de déformation sur le bord}, we have
\begin{eqnarray*}
2[\mathrm{\textbf{D}}(\boldsymbol{\varphi})\textbf{\textit{n}}]_{\tau}+\alpha\boldsymbol{\varphi}_{\tau}&=&\left( \frac{\partial \boldsymbol{\varphi}}{\partial \textbf{\textit{n}}}\right)_{\tau}-\Lambda \boldsymbol{\varphi}+\alpha\boldsymbol{\varphi}_{\tau}.
\end{eqnarray*}
Since
$\dfrac{\partial\boldsymbol{\varphi}}{\partial\textbf{\textit{n}}}=\Lambda\,\boldsymbol{\mu}-\textbf{\textit{n}}\,\mathrm{div}\,_{\Gamma}\boldsymbol{\mu}_{\tau}-\alpha\boldsymbol{\mu}_{\tau}$ on  $\Gamma$ and  $\boldsymbol{\varphi}=\boldsymbol{\mu}_{\tau}$ on $\Gamma$, then we have
\begin{eqnarray*}
\Big(\dfrac{\partial\boldsymbol{\varphi}}{\partial\textbf{\textit{n}}}\Big)_{\tau}&=&(\Lambda\,\boldsymbol{\mu})_{\tau}-\alpha\boldsymbol{\mu}_{\tau}\\
&=& \Lambda\,\boldsymbol{\varphi}-\alpha\boldsymbol{\varphi}_{\tau}\\
\end{eqnarray*}

\noindent We deduce that $2[\mathrm{\textbf{D}}(\boldsymbol{\varphi})\textbf{\textit{n}}]_{\tau}+\alpha\boldsymbol{\varphi}_{\tau}=0$ on $\Gamma$. Which implies that the function $\boldsymbol{\varphi}$ belongs to ${S}^{}_{k+1}(\Omega)$ and satisfies:
\begin{eqnarray*}
\begin{cases}
\boldsymbol{\varphi}=\mu_{\tau}\,\,\,\qquad\qquad\quad\qquad\qquad\quad\qquad\text{ on }\,\,\,\Gamma ,\\
\nabla\,\boldsymbol{\varphi}\cdot\textbf{\textit{n}}=\Lambda\,\mu-\textbf{\textit{n}}\,\mathrm{div}\,_{\Gamma}\mu_{\tau}-\alpha\boldsymbol{\mu}_{\tau}\qquad\text{ on }\,\,\, \Gamma.
\end{cases}
\end{eqnarray*}
\noindent In addition, we have the following estimate
\begin{eqnarray}\label{estimation trace}
\Vert \boldsymbol{\varphi}\Vert_{W^{2,2}_{k+1}(\Omega)}\leqslant C\Vert{\boldsymbol{\mu}}_{\tau}\Vert_{H^{3/2}(\Gamma)}\leqslant C\Vert\boldsymbol{\mu}\Vert_{H^{3/2}(\Gamma)}.
\end{eqnarray}

\noindent Consequently,
\begin{eqnarray*}
\Big\vert \left\langle 2[\mathrm{\textbf{D}}(\textbf{\textit{u}})\textbf{\textit{n}}]_{\tau}+\alpha\textbf{\textit{u}}_{\tau},\boldsymbol{\mu} \right\rangle_{\Gamma}\Big\vert 
&=&\Big\vert \left\langle 2[\mathrm{\textbf{D}}(\textbf{\textit{u}})\textbf{\textit{n}}]_{\tau}+\alpha\textbf{\textit{u}}_{\tau},\boldsymbol{\mu}_{\tau} \right\rangle_{\Gamma}\Big\vert
=\Big\vert \left\langle 2 [\mathrm{\textbf{D}}(\textbf{\textit{u}})\textbf{\textit{n}}]_{\tau}+\alpha\textbf{\textit{u}}_{\tau},\boldsymbol{\varphi} \right\rangle_{\Gamma}\Big\vert\\
&\leqslant & \Big\vert\left\langle\Delta\,\textbf{\textit{u}},\boldsymbol{\varphi} \right\rangle _{\Omega}\Big\vert+\Big\vert\displaystyle\int_{\Omega}\textbf{\textit{u}}\cdot\Delta\,\boldsymbol{\varphi} d\textbf{x}-2\langle[\textbf{D}(\boldsymbol{\varphi})\textbf{\textit{n}}]\cdot\textbf{\textit{n}},\textbf{\textit{u}}\cdot\textbf{\textit{n}}\rangle_{\Gamma}\Big\vert\\
&\leqslant & \Vert\Delta\,\textbf{\textit{u}}\Vert_{(T_{k+1}^{}(\Omega))'}\Vert \boldsymbol{\varphi}\Vert_{T_{k+1}^{}(\Omega)}
+\Vert\textbf{\textit{u}}\Vert_{W^{0,2}_{-k-1}(\Omega)}\Vert \boldsymbol{\varphi}\Vert_{W^{2,2}_{k+1}(\Omega)}\\
&\leqslant & C\Vert\textbf{\textit{u}}\Vert_{{H}_{-k-1}(\Delta ;\Omega)}\Vert\boldsymbol{\varphi}\Vert_{W^{2,2}_{k+1}(\Omega)}
\end{eqnarray*}

\noindent Thus, using \eqref{estimation trace}, we obtain that for any $\textbf{\textit{u}}\in \mathcal{D}(\overline{\Omega})$:

\begin{eqnarray*}
\Vert 2[\mathrm{\textbf{D}}(\textbf{\textit{u}})\textbf{\textit{n}}]_{\tau}+\alpha\textbf{\textit{u}}_{\tau}\Vert_{H^{-3/2}(\Gamma)}\leqslant C\Vert \textbf{\textit{u}}\Vert_{{H}^{}_{-k-1}(\Delta ;\Omega)}.
\end{eqnarray*}

\noindent Therefore, the linear mapping $\Theta\, :\, \textbf{\textit{u}}\rightarrow {2[\mathrm{\textbf{D}}(\textbf{\textit{u}})\textbf{\textit{n}}]_{\tau}+\alpha\textbf{\textit{u}}_{\tau}}$ defined in 
$\mathcal{D}(\overline{\Omega})$ is continuous for the norm of ${H}^{}_{-k-1}(\Delta ;\Omega)$. Finally, by density of $\mathcal{D}(\overline{\Omega})$ in ${H}^{}_{-k-1}(\Delta ;\Omega)$, we can extend this mapping from ${H}_{-k-1}^{}(\Delta ;\Omega)$ into $H^{-3/2}(\Gamma)$ and formula~\eqref{Green formula very weak p=2} holds.
\end{proof}
%\section{Strong solution for the Stokes equations with Navier boundary conditions}
%\noindent We recall here some results concerning the strong solutions for the exterior Stokes problem with Navier slip boundary conditions:
%\begin{eqnarray*}
%(\mathcal{S}_{T})
%\begin{cases}
%&-\Delta\,\textbf{\textit{u}}+\nabla\,{\pi}=\textbf{\textit{f}}\quad\text{and}\quad
%\mathrm{div}\,\textbf{\textit{u}}={\chi}\quad\text{in}\quad\Omega,\\
%&\textbf{\textit{u}}\cdot \textbf{\textit{n}}=g\quad\text{and}\quad
%2[\mathrm{\textbf{D}}(\textbf{\textit{u}})\textbf{\textit{n}}]_{\tau}+\alpha\textbf{\textit{u}}_{\tau}=\textbf{\textit{h}}\quad\text{ on }\quad\Gamma.
%\end{cases}
%\end{eqnarray*}
%In this part, we recall the well-posedness of strong solutions in $W_{k+1}^{2}(\Omega)\times W_{k+1}^{1}(\Omega)$ with $k\in \Z$. These results can be found in~\cite{DMR-2019}. In order to deal with the uniqueness issues, let us first introduce the kernel of problem $(\mathcal{S}_{T})$, for $k\in\Z$:\\
%\small{\begin{eqnarray*}
%\mathcal{N}_{k}(\Omega)=\Big\lbrace  (\textbf{\textit{u}},{\pi})\in W^{1}_{k}(\Omega)\times{{W}}^{0}_{k}(\Omega);\, -\Delta\,\textbf{\textit{u}}+\nabla\,{\pi}=0,\, \mathrm{div}\,\textbf{\textit{u}}=0\text{ in }\,\Omega\text{ and }\textbf{\textit{u}}\cdot \textbf{\textit{n}}=0,\,\,2[\mathrm{\textbf{D}}(\textbf{\textit{u}})\textbf{\textit{n}}]_{\tau}+\alpha\textbf{\textit{u}}_{\tau}=0\,\text{ on }\,\Gamma\Big\rbrace.
%\end{eqnarray*}}
\section{Very weak solution}\label{solution-trés-faible-p=2}
\noindent In this section, we consider the Stokes problem with Navier slip boundary conditions:
\begin{eqnarray*}
(\mathcal{S}_{T})
\begin{cases}
&-\Delta\,\textbf{\textit{u}}+\nabla\,{\pi}=\textbf{\textit{f}}\quad\text{and}\quad
\mathrm{div}\,\textbf{\textit{u}}={\chi}\quad\text{in}\quad\Omega,\\
&\textbf{\textit{u}}\cdot \textbf{\textit{n}}=g\quad\text{and}\quad
2[\mathrm{\textbf{D}}(\textbf{\textit{u}})\textbf{\textit{n}}]_{\tau}+\alpha\textbf{\textit{u}}_{\tau}=\textbf{\textit{h}}\quad\text{ on }\quad\Gamma.
\end{cases}
\end{eqnarray*}
\noindent Our aim is to investigate the existence and the uniqueness of a very weak solutions belongs to $W^{0,2}_{-k-1}(\Omega)\times W^{-1,2}_{-k-1}(\Omega)$, where $k\in \Z$ for the Stokes problem $(\mathcal{S}_T)$. The main idea consists in the use of a duality argument with the strong solutions obtained in~\cite{DMR-2019}.  We assume that $\alpha$ is a positive function that belongs to $W^{1,\infty}(\Gamma)$.\\

\noindent Before stating the theorem of the existence and the uniqueness of the very weak solution for Stokes problem, we need to introduce the following null spaces for $s\in \left\lbrace 1,2\right\rbrace $ and $k \in \Z $:
\begin{align*}
\mathcal{N}_{k}^{s,2}(\Omega)=\left\lbrace (\textbf{\textit{u}},\pi) \in {W}_{k}^{s,2}(\Omega) \times W_{k}^{s-1,2}(\Omega);
\,\, T(\textbf{\textit{u}},\pi)=(\boldsymbol{0},0) \,\,\,\mathrm{in} \,\,\,\Omega  \quad\mathrm{and}\quad \textbf{\textit{u}}\cdot\textbf{\textit{n}}=0,\quad 2[\mathrm{\textbf{D}}(\textbf{\textit{u}})\textbf{\textit{n}}]_{\tau}+\alpha\textbf{\textit{u}}_{\tau}=\boldsymbol{0}
\,\,\,\text{on}\,\,\,\Gamma \right\rbrace,
\end{align*}
with 
\begin{equation*}
 T(\textbf{\textit{u}},\pi)=(-\Delta\,\textbf{\textit{u}}+\nabla\,\pi,\,\,\mathrm{div}\,\textbf{\textit{u}}).
\end{equation*}
For all $k\in\Z$, we introduce the following spaces:
\begin{eqnarray*}
N_{k}=\Big\lbrace(\boldsymbol{\lambda},\mu)\in \mathcal{P}_{k}\times\mathcal{P}^{\Delta}_{k-1},\,\,\,-\Delta\,\boldsymbol{\lambda}+\nabla\,\mu
=\boldsymbol{0}\,\,\text{ and }\,\, \mathrm{div}\,\boldsymbol{\lambda}=0\Big\rbrace 
\end{eqnarray*}
that is the null space of the Stokes operator in the whole space $\R^3$, we recall that by agreement on the notation $\mathcal{P}_{k}$, the space $N_{k}=\lbrace (\boldsymbol{0},0)\rbrace$ when $k<0$ and $N_{0}=\mathcal{P}_{0}\times\lbrace 0\rbrace$.\\

\noindent The next proposition characterizes the kernel of $\mathcal{N}_{k}^{1,2}(\Omega)$:

\begin{prop}\label{caracterisation du noyau dans non borne}
Suppose that $\Omega$ is of class $\mathcal{C}^{1,1}$ and assume $k\in\Z$. 
\begin{itemize}
\item[$\bullet$] If $k\ge0$, then $\mathcal{N}_{k}^{1,2}(\Omega)=\big\lbrace (\boldsymbol{0},0)\big\rbrace$.
\item[$\bullet$] If $k<0$, then $\mathcal{N}_{k}^{1,2}(\Omega)=\left\lbrace  (\textbf{\textit{v}}-\boldsymbol{\lambda},{\theta}-\mu); \quad (\boldsymbol{\lambda},\mu)\in N_{-k-1}\right\rbrace $, 
where $(\textbf{\textit{v}},{\theta})\in W^{1,2}_{0}(\Omega)\times{{L}}^{2}_{}(\Omega)$ is the unique solution of the following problem:
\begin{equation}
\label{problem stokes du noyau}
\begin{cases}
-\Delta\,\textbf{\textit{v}}+\nabla\,{\theta}=\boldsymbol{0}\quad\text{and}\quad
\mathrm{div}\,\textbf{\textit{v}}=0\quad\text{in}\quad\Omega,\\[4pt]
\textbf{\textit{v}}\cdot \textbf{\textit{n}}=\boldsymbol{\lambda}\cdot\textbf{\textit{n}}\quad\text{and}\quad
2[\mathrm{\textbf{D}}(\textbf{\textit{v}})\textbf{\textit{n}}]_{\tau}+\alpha\textbf{\textit{v}}_{\tau}=
2[\mathrm{\textbf{D}}(\boldsymbol{\lambda})\textbf{\textit{n}}]_{\tau}+\alpha\boldsymbol{\lambda}_{\tau}\quad\text{on}\quad\Gamma.
\end{cases}
\end{equation}
\end{itemize}
\end{prop}
\noindent The proof of this Proposition can be found in~\cite{DMR-2019}.\\

\noindent The next theorem states an existence, uniqueness and regularity result for problem $(\mathcal{S}_{T})$ (for instance see~\cite{DMR-2019}).
\begin{theo}\label{solution forte}
Suppose that $\Omega$ is of class $\mathcal{C}^{2,1}$ and let $k\in\Z$.
Assume that $\textbf{\textit{f }}\in W^{0,2}_{k+1}(\Omega)$, $g\in H^{3/2}(\Gamma)$,  $\chi\in W^{1,2}_{k+1}(\Omega)$ and $\textbf{\textit{h }}\in H^{1/2}(\Gamma)$ 
satisfying \,\, $\textbf{\textit{h}}\cdot\textbf{\textit{n}}=0$ on $\Gamma$. Assume moreover that the following compatibility condition is satisfied
\begin{equation}\label{condition de compatibilite}
\forall (\boldsymbol{\xi},\eta)\in \mathcal{N}_{-k}^{1,2}(\Omega),\quad
\displaystyle\int_{\Omega} \textbf{\textit{f }}\cdot\boldsymbol{\xi}d\textbf{\textit{x}}-\displaystyle\int_{\Omega}
\chi\,{\eta}d\textbf{\textit{x}}=\langle g , 2[\textrm{\textbf{D}}(\boldsymbol{\xi})\textbf{\textit{n}}]\cdot\textbf{\textit{n}}-\eta\rangle_{\Gamma}-\langle\textbf{\textit{h}}, 
\boldsymbol{\xi}\rangle_{\Gamma}.\quad 
\end{equation}
\noindent Then, the Stokes problem $(\mathcal{S}_{T})$ has a solution
$(\textbf{\textit{u}},{\pi})\in W^{2,2}_{k+1}(\Omega)\times {W}^{1,2}_{k+1}(\Omega)$ unique up to an element of $\mathcal{N}_{k}^{1,2}(\Omega)$ and we have:
$$
\inf _{(\boldsymbol{\lambda},\mu)\in\mathcal{N}_{k}^{1,2}(\Omega)} \left( \Vert \textbf{\textit{u}}+\boldsymbol{\lambda}\Vert_{W^{2,2}_{k+1}(\Omega)}
+\Vert {\pi}+{\mu}\Vert_{{W}^{1,2}_{k+1}(\Omega)}\right)  
\leqslant  C\Big(\Vert \textbf{\textit{f }}\Vert_{W^{0,2}_{k+1}(\Omega)}+\Vert\chi\Vert_{W^{1,2}_{k+1}(\Omega)}+\Vert g\Vert_{H^{3/2}(\Gamma)}
+\Vert \textbf{\textit{h}}\Vert_{H^{1/2}(\Gamma)}\Big).$$
\end{theo}
%\noindent The proofs of  Theorem~\ref{solution forte} can be found in~\cite{DMR-2019}.\\
\noindent The following lemma proves the identity between some null spaces.
 
\begin{lemm}\label{1.montrer egalite de noyaux}
Assume that \, $\Omega$\,  is of class $C^{2,1}$,\,\, $\alpha\in W^{1,\infty}(\Gamma)$ and $k\in\Z$ then we have the following identity:
\begin{equation*}
\mathcal{N}_{-k+1}^{2,2}(\Omega)=\mathcal{N}_{-k}^{1,2}(\Omega).
\end{equation*}
\end{lemm}
\begin{proof}
For the proof of this lemma, we shall apply a technique used in~\cite{Meslameni-2013} for the Stokes equation with the Dirichlet boundary conditions. The proof falls into two parts: 
\begin{itemize}
\item[$\bullet$] First inclusion: Using the imbedding~\eqref{inclusion.sobolev}, then we have $ \mathcal{N}_{-k+1}^{2,2}(\Omega)\subset\mathcal{N}_{-k}^{1,2}(\Omega).$\\
\item[$\bullet$] Second inclusion: Let $(\textbf{\textit{u}},\pi) \in \textbf{\textit{W}}_{-k}^{1,2}(\Omega) \times W_{-k}^{0,2}(\Omega)$ such that
\begin{eqnarray*}
\begin{cases}
&-\Delta\,\textbf{\textit{u}}+\nabla\,{\pi}=\textbf{\textit{0}}\quad\text{and}\quad
\mathrm{div}\,\textbf{\textit{u}}={0}\quad\text{in}\quad\Omega,\\
&\textbf{\textit{u}}\cdot \textbf{\textit{n}}=0\quad\text{and}\quad
2[\mathrm{\textbf{D}}(\textbf{\textit{u}})\textbf{\textit{n}}]_{\tau}+\alpha\textbf{\textit{u}}_{\tau}=\textbf{\textit{0}}\quad\text{ on }\quad\Gamma.
\end{cases}
\end{eqnarray*}
Note that if $\textbf{\textit{u}} \in \textbf{\textit{W}}_{-k}^{1,2}(\Omega)$ and $-\Delta\,\textbf{\textit{u}}+\nabla\,\pi=\boldsymbol{0}$ in $\Omega$ with $\pi \in W_{-k}^{0,2}(\Omega)$, we obtain $\textbf{\textit{u}}\in {H}_{-k}^2(\Delta;\Omega)$ then thanks \cite[Lemma 2.5]{DMR-2019} we have $\textbf{\textit{u}}\mapsto 2[\textbf{D}(\textbf{\textit{u}})\textbf{\textit{n}}]_{\tau}+\alpha\textbf{\textit{u}}_{\tau}$ belongs to $H^{-1/2}(\Gamma)$, and if $\mathrm{div}\,\textbf{\textit{u}}=0$ in $\Omega$, then $\textbf{\textit{u}}\cdot \textbf{\textit{n}} \in H^{\,1/2}(\Gamma)$. That means that the boundary conditions makes sense.\\

\noindent Now, let us introduce the following partition of unity:
\begin{equation}
 \label{partition de l'unite}
 \begin{split}
&{\varphi},\,\psi\in\mathcal{C}^\infty(\R^3),\quad 0\le\varphi,\,\psi\le1,\quad \varphi+\psi=1\quad\text{in}\quad\R^3,\\ 
&\varphi=1\quad\text{in}\quad B_R,\quad\text{supp }\varphi\subset B_{R+1}.
\end{split}
\end{equation}
Let $\Omega_{R+1}$ denote the intersection $\Omega\cap B_{R_{}+1}$. Then, we can write
\begin{equation*}
 \textbf{\textit{u}}=\varphi\,\textbf{\textit{u}}+\psi\,\textbf{\textit{u}},\quad\quad \pi=\varphi\,\pi+\psi\,\pi.
\end{equation*}
 The pair $(\textbf{\textit{u}},\pi)$ has an extension $(\widetilde{\textbf{\textit{u}}},\widetilde{\pi})$ that belongs to $W^{1,2}_{-k}(\R^3)\times{{W}}^{0,2}_{-k}(\R^3)$. It suffices to prove that $(\varphi \,\textbf{\textit{u}},\varphi\,\pi)$ belongs to $H^{2}(\Omega_{R_{}+1})\times H^{1}(\Omega_{R_{}+1})$ and that $(\psi\,\tilde{\textbf{\textit{u}}},\psi\,\tilde{\pi})$ belongs to $\textbf{\textit{W}}_{-k+1}^{2,2}(\R^{3})\times W^{1,2}_{-k+1}(\R^{3}) $.\\

To that end, consider first
\begin{equation}
\label{stokes dans R^3 pour psiu}
 -\Delta\,({\psi}\widetilde{\textbf{\textit{u}}})+\nabla\,({\psi}\widetilde{\pi})=\textbf{\textit{f}}_{1}\quad \mathrm{and}\quad \mathrm{div}\,({\psi}\widetilde{\textbf{\textit{u}}})=\chi_{1} 
 \quad\text{in}\quad\R^3,
 \end{equation}
 where
$$\textbf{\textit{f}}_{1}=-(2\nabla\widetilde{\textbf{\textit{u}}}\nabla{\psi}+\textbf{\textit{u}}\Delta\,{\psi})+\widetilde{\pi}\nabla\,{\psi}\quad\text{ and } 
\quad \chi_{1}=\widetilde{\textbf{\textit{u}}}\cdot\nabla\,\psi.
$$
We easily see that $\textbf{\textit{f}}_{1}$ and $\chi_1$ have bounded supports and belong to $L_{loc}^{2}(\R^3)\times W_{loc}^{1,2}(\R^3)$.
As a consequence, $(\textbf{\textit{f}}_{1},\chi_1)$ belongs to $W^{0,2}_{-k+1}(\R^3)\times W^{1,2}_{-k+1}(\R^3)$, we need to show that $\textit{\textbf{f}}_1$ and $\chi_1$ satisfies the following compatibility condition:
\begin{equation}
\label{CCR^3}
\forall (\boldsymbol{\lambda},\mu)\in N_{-k-1},\quad \langle\textbf{\textit{f}}_1,\boldsymbol{\lambda}\rangle_{W^{0,2}_{-k+1}(\R^3)\times W^{0,2}_{k-1}(\R^3)} 
-\langle \chi_1,\mu\rangle_{W^{1,2}_{-k+1}(\R^3)\times W^{-1,2}_{k-1}(\R^3)}=0,
\end{equation}
For this, using the same calculation in the proof of \cite[Theorem~$3.7$]{DMR-2019}. We obtain~\eqref{CCR^3}.\\

\noindent Therefore, it follows from \cite[Theorem 3.9]{Alliot_M3AS_1999}, that there exists a unique solution 
$(\widetilde{\textbf{\textit{v}}},\widetilde{q})\in(\textbf{\textit{W}}^{2,2}_{-k+1}(\R^3)\times {W}^{1,2}_{-k+1}(\R^3))$ satisfying the following Stokes problem:
\begin{eqnarray*}
-\Delta\,\widetilde{\textbf{\textit{v}}}+\nabla\,\widetilde{q}=\textit{\textbf{f}}_1\quad \text{and} 
\quad \mathrm{div}\,\widetilde{\textbf{\textit{v}}}=\chi_1\quad\text{in}\quad\R^3.
\end{eqnarray*}
It follows that $(\widetilde{\textbf{\textit{v}}}-\psi\widetilde{\textbf{\textit{u}}},\widetilde{q}-\psi\widetilde{\pi})$ belongs to $N_{k-1}$. Since $N_{k-1}\subset W^{2,2}_{-k+1}(\R^3)\times W^{1,2}_{-k+1}(\R^3)$, then there exist $(P,Q)$ belongs to $
 W^{2,2}_{-k+1}(\R^3)\times W^{1,2}_{-k+1}(\R^3)$ such that $(\widetilde{\textbf{\textit{v}}}-\psi\widetilde{\textbf{\textit{u}}},\widetilde{q}-\psi\widetilde{\pi})=(P,Q)$. Consequently, $(\psi\widetilde{\textbf{\textit{u}}},\psi\widetilde{\pi})$ belongs to $W^{2,2}_{-k+1}(\R^3)\times W^{1,2}_{-k+1}(\R^3)$.\\ 

 \noindent Consider now the system
$$
 -\Delta\,({\varphi}\widetilde{\textbf{\textit{u}}})+\nabla\,({\varphi}\widetilde{\pi})
 =\textbf{\textit{f}}_{2}\quad \mathrm{and}\quad \mathrm{div}\,({\varphi}\widetilde{\textbf{\textit{u}}})=\chi_{2},
$$
\noindent where $\textbf{\textit{f}}_2$ and $\chi_2$ have similar expressions as $\textbf{\textit{f}}_1$ and $\chi_1$ 
with $\psi$ remplaced by $\varphi$. It is easy to check that $(\textbf{\textit{f}}_{2},\chi_{2})$ belongs to $L^{2}(\Omega_{R+1})\times H^{1}(\Omega_{R+1})$. Then the regularity results for the Stokes problem with Navier boundary conditions in a bounded domain of class $\mathcal{C}^{2,1}$ (see~~\cite{Ahmad_2014}), allow to prove that
the following Stokes problem:
\begin{eqnarray*}
\begin{cases}
&-\Delta\,\varphi\textbf{\textit{u}}+\nabla\,\varphi{\pi}=\textbf{\textit{f}}_{2}\quad\text{and}\quad
\mathrm{div}\,\varphi\textbf{\textit{u}}={\chi}_2\quad\text{in}\quad\Omega_{R+1},\\
&\varphi\textbf{\textit{u}}\cdot \textbf{\textit{n}}=0\quad\text{and}\quad
2[\mathrm{\textbf{D}}(\varphi\textbf{\textit{u}})\textbf{\textit{n}}]_{\tau}+\alpha(\varphi\textbf{\textit{u}})_{\tau}=\textbf{\textit{0}}\quad\text{ on }\quad\Gamma.
\\
&\varphi\textbf{\textit{u}}\cdot \textbf{\textit{n}}=0\quad\text{and}\quad
2[\mathrm{\textbf{D}}(\varphi\textbf{\textit{u}})\textbf{\textit{n}}]_{\tau}+\alpha(\varphi\textbf{\textit{u}})_{\tau}=\textbf{\textit{0}}\quad\text{ on }\quad\partial\,B_{R+1},
\end{cases}
\end{eqnarray*}
has a solution $({\varphi}\widetilde{\textbf{\textit{u}}},{\varphi}\widetilde{{\pi}})$ belongs to $ H^{2}(\Omega_{R+1})\times H^{1}(\Omega_{R+1})$ which also implies 
that $({\varphi}\widetilde{\textbf{\textit{u}}},{\varphi}\widetilde{{\pi}})$ belongs to $W^{2,2}_{-k+1}(\R^3)\times W^{1,2}_{-k+1}(\R^3)$.\\
\noindent Consequently, the pair $(\textit{\textbf{u}},\pi)$ belongs to $W^{2,2}_{-k+1}(\Omega)\times W^{1,2}_{-k+1}(\Omega)$ and thus $ \mathcal{N}_{-k}^{1,2}(\Omega)\subset\mathcal{N}_{-k+1}^{2,2}(\Omega)$.
\end{itemize} 
\end{proof}

\noindent Now, we begin by giving the definition of very weak solutions.
\begin{defi}\label{définition formulation faible p=2}
Given $\textbf{\textit{f}}$, $\chi$, $g$ and $\textbf{\textit{h}}$ with
\begin{eqnarray*}
\textbf{\textit{f}}\in (\textbf{\textit{T}}^{2}_{k+1}(\Omega))',\quad \chi\in W^{0,2}_{-k}(\Omega),\quad g\in H^{-1/2}(\Gamma),\quad \textbf{\textit{h}}\in H^{-3/2}(\Gamma)
\end{eqnarray*}
such that $\textbf{\textit{h}}\cdot\textbf{\textit{n}}=0$ on $\Gamma$, a pair $(\textbf{\textit{u}},\pi)\in W^{0,2}_{-k-1}(\Omega)\times W^{-1,2}_{-k-1}(\Omega)$ is called very weak solution of $(\mathcal{S}_{T})$ if and only if, for any $(\boldsymbol{\varphi},q)\in {S}_{k+1}^2(\Omega)\times W^{1,2}_{k+1}(\Omega)$, the relations
\begin{eqnarray}\label{formulation faible I p=2}
 &&-\displaystyle\int_{\Omega} \textbf{\textit{u}}\cdot \Delta\,\boldsymbol{\varphi} d\textbf{x}-\langle\pi,\mathrm{div}\,\boldsymbol{\varphi}\rangle_{W^{-1,2}_{-k-1}(\Omega)\times\mathring{ W}^{1,2}_{k+1}(\Omega)}
\\
\nonumber &&=\langle\textbf{\textit{f}},\boldsymbol{\varphi}\rangle_{(\textbf{\textit{T}}_{k+1}^{2}(\Omega))'\times \textbf{\textit{T}}_{k+1}^{2}(\Omega)}
+\langle\textbf{\textit{h}},\boldsymbol{\varphi}\rangle_{ H^{-3/2}(\Gamma)\times  H^{3/2}(\Gamma)}-\langle g,2[\textbf{D}(\boldsymbol{\varphi})\textbf{\textit{n}}]\cdot\textbf{\textit{n}}\rangle_{H^{-1/2}(\Gamma)\times H^{1/2}(\Gamma)}
\end{eqnarray}
and
\begin{eqnarray}\label{formulation faible II p=2}
\displaystyle\int_{\Omega} \textbf{\textit{u}}\cdot \nabla\,q d\textbf{x}=-\displaystyle\int_{\Omega} \chi q d\textbf{x}+\langle g,q\rangle_{H^{-1/2}(\Gamma)\times H^{1/2}(\Gamma)}
\end{eqnarray}
are satisfied.
\end{defi}
\noindent Note that ${S}^{}_{k+1}(\Omega)$ is included in $\textbf{\textit{T}}^{}_{k+1}(\Omega)$, which insures that all brackets are well defined
\begin{prop}\label{proposition equivalence p=2}
Under the assumptions of Definition \ref{définition formulation faible p=2}, the following two statements are equivalent:\\
$i)$\quad $(\textbf{\textit{u}},\pi)\in W^{0,2}_{-k-1}(\Omega)\times  W^{-1,2}_{-k-1}(\Omega)$ is a very weak solution of $(\mathcal{S}_{T})$.\\
$ii)$\quad $(\textbf{\textit{u}},\pi)\in W^{0,2}_{-k-1}(\Omega)\times  W^{-1,2}_{-k-1}(\Omega)$ satisfies the problem $(\mathcal{S}_{T})$ in the sense of distributions.
\end{prop}
\begin{proof}\quad\\
\begin{itemize}
\item[i)$\Rightarrow$ ii)] Let $(\textbf{\textit{u}},\pi)\in W^{0,2}_{-k-1}(\Omega)\times  W^{-1,2}_{-k-1}(\Omega)$ a very weak solution of $(\mathcal{S}_{T})$. Therefore, it follows immediately from~\eqref{formulation faible I p=2} and~\eqref{formulation faible II p=2}, that for any $\boldsymbol{\varphi}\in \mathcal{D}(\Omega)$, we have 
\begin{eqnarray*}
-\Delta\,\textbf{\textit{u}}+\nabla\,\pi=\textbf{\textit{f}}\quad\text{ and }\quad\mathrm{div}\,\textbf{\textit{u}}=\chi\quad\text{ in }\,\,\, \Omega.
\end{eqnarray*}
It remains now to prove that the Navier boundary conditions are satisfied. Let us first prove that $\textbf{\textit{u}}\cdot\textbf{\textit{n}}=g$ on $\Gamma$. For this, we consider the equation $\mathrm{div}\,\textbf{\textit{u}}=\chi$ in $\Omega$. We multiply this equation by $q\in W^{1,2}_{k+1}(\Omega)$ and compare with \eqref{formulation faible II p=2}. We get
\begin{eqnarray*}
\langle\textbf{\textit{u}}\cdot\textbf{\textit{n}},q\rangle_{H^{-1/2}(\Gamma)\times H^{1/2}(\Gamma)}=\langle g,q\rangle_{H^{-1/2}(\Gamma)\times H^{1/2}(\Gamma)}.
\end{eqnarray*}
Which yields $\textbf{\textit{u}}\cdot\textbf{\textit{n}}=g \in H^{-1/2}(\Gamma)$.\\
Next, writing
\begin{eqnarray*}
-\Delta\,\textbf{\textit{u}}=\textbf{\textit{f}}-\nabla\,\pi.
\end{eqnarray*}
Using Lemma~\ref{caractérisation T^2}, we deduce that $\textbf{\textit{u}}$ belongs to ${H}_{-k-1}(\Delta\, ; \Omega)$. It follows from Lemma~\ref{trace Navier p=2} that $2[\textbf{D}(\textbf{\textit{u}})\textbf{\textit{n}}]_{\tau}+\alpha\textbf{\textit{u}}_{\tau}\in H^{-3/2}(\Gamma)$. Applying~\eqref{Green formula very weak p=2} and~\eqref{formule d'integration dans T^2} enables us to write: for any $\boldsymbol{\varphi}\in {S}_{k+1}^2(\Omega)$,
\begin{eqnarray}\label{integration dans S^2}
&& -\displaystyle\int_{\Omega} \textbf{\textit{u}}\cdot \Delta\,\boldsymbol{\varphi} d\textbf{x}-\langle 2[\textbf{D}(\textbf{\textit{u}})\textbf{\textit{n}}]_{\tau}+\alpha\textbf{\textit{u}}_{\tau},\boldsymbol{\varphi}\rangle_{H^{-3/2}(\Gamma)\times H^{3/2}(\Gamma)}\\
\nonumber&&+\langle 2[\textbf{D}({\varphi})\textbf{\textit{n}}]\cdot\textbf{\textit{n}},\textbf{\textit{u}}\cdot\textbf{\textit{n}}\rangle_{H^{1/2}(\Gamma)\times H^{-1/2}(\Gamma}=\langle\textbf{\textit{f}},\boldsymbol{\varphi}\rangle_{(\textbf{\textit{T}}^{}_{k+1}(\Omega))'\times \textbf{\textit{T}}^{}_{k+1}(\Omega)}+\langle\pi,\mathrm{div}\,\boldsymbol{\varphi}\rangle_{W^{-1,2}_{-k-1}(\Omega)\times \mathring{W}^{1,2}_{k+1}(\Omega)}.
\end{eqnarray}

\noindent Comparing~\eqref{formulation faible I p=2} with~\eqref{integration dans S^2}, we get, for any $\boldsymbol{\varphi}\in {S}_{k+1}(\Omega)$
{\begin{eqnarray*}
\langle 2[\textbf{D}(\textbf{\textit{u}})\textbf{\textit{n}}]_{\tau}+\alpha\textbf{\textit{u}}_{\tau},\boldsymbol{\varphi}\rangle_{H^{-3/2}(\Gamma)\times H^{3/2}(\Gamma)}=\langle\textbf{\textit{h}},\boldsymbol{\varphi}\rangle_{H^{-3/2}(\Gamma)\times H^{3/2}(\Gamma)}.
\end{eqnarray*}}

 Let $\boldsymbol{\mu} \in H^{3/2}(\Gamma)$, there exists $\boldsymbol{\varphi}\in {S}_{k+1}(\Omega)$ such that $\boldsymbol{\varphi}=\boldsymbol{\mu}_{\tau}$ on $\Gamma$ (see proof of Lemma~\ref{trace Navier p=2}).
Consequently,
\begin{eqnarray*}
\langle 2[\textbf{D}(\textbf{\textit{u}})\textbf{\textit{n}}]_{\tau}+\alpha\textbf{\textit{u}}_{\tau},\boldsymbol{\mu} \rangle_{H^{-3/2}(\Gamma)\times H^{3/2}(\Gamma)}=\langle\textbf{\textit{h}},\boldsymbol{\mu}\rangle_{H^{-3/2}(\Gamma)\times H^{3/2}(\Gamma)}
\end{eqnarray*}
and we deduce that $2[\textbf{D}(\textbf{\textit{u}})\textbf{\textit{n}}]_{\tau}+\alpha\textbf{\textit{u}}_{\tau}=\textbf{\textit{h}}$ on $\Gamma$.\\

\item[ii)$\Rightarrow$i)] Suppose that $(\textbf{\textit{u}},\pi)\in W^{0,2}_{-k-1}(\Omega)\times  W^{-1,2}_{-k-1}(\Omega)$ satisfies the problem $(\mathcal{S}_{T})$ in the sense of distributions. Then, we have
\begin{eqnarray*}
-\Delta\textbf{\textit{u}}+\nabla\,\pi=\textbf{\textit{f}}\quad\text{ and }\quad\mathrm{div}\,\textbf{\textit{u}}=\chi\quad\text{in }\,\,\mathcal{D}'(\Omega),
\end{eqnarray*}
which implies that, for any $\boldsymbol{\varphi}\in\mathcal{D}(\Omega)$, we have
\begin{eqnarray}\label{f au sens de distrubtion p=2}
\langle-\Delta\textbf{\textit{u}}+\nabla\,\pi,\boldsymbol{\varphi}\rangle_{\mathcal{D}'(\Omega)\times \mathcal{D}(\Omega)}=\langle\textbf{\textit{f}},\boldsymbol{\varphi}\rangle_{\mathcal{D}'(\Omega)\times \mathcal{D}(\Omega)}.
\end{eqnarray}
Since $\mathcal{D}(\Omega)$ is dense in $T^{}_{k+1}(\Omega)$, then~\eqref{f au sens de distrubtion p=2} is still valid for any $\boldsymbol{\varphi}\in T^{}_{k+1}(\Omega)$. In particular, \eqref{f au sens de distrubtion p=2} is still valid for any $\boldsymbol{\varphi}\in{S}_{k+1}(\Omega)\subset T^{}_{k+1}(\Omega)$. Since $\nabla\,\pi\in (T^{}_{k+1}(\Omega))'$ and $-\Delta\,\textbf{\textit{u}}=\textbf{\textit{f}}-\nabla\,\pi\in (T^{}_{k+1}(\Omega))'$. Then according to \eqref{formule d'integration dans T^2} and \eqref{Green formula very weak p=2}, it is clear that~\eqref{formulation faible I p=2} holds.\\ 
Now from the equation $\mathrm{div}\,\textbf{\textit{u}}=\chi$ in $\Omega$, we can deduce that for any $q\in \mathcal{D}(\Omega)$
\begin{eqnarray*}
\displaystyle\int_{\Omega}\mathrm{div}\,\textbf{\textit{u}}q d\textbf{x}=\displaystyle\int_{\Omega}\chi q d\textbf{x}.
\end{eqnarray*}

Using the Green's formula~\eqref{FG1}, we can deduce~\eqref{formulation faible II p=2}. 
\end{itemize}
\end{proof}

\noindent Now, we are in position to prove the existence and uniqueness of a very weak solution for Stokes problem $(\mathcal{S}_{T})$. 
\begin{theo}\label{solutions très faibles p=2}
Suppose that $g=0$.
Given any $\textbf{\textit{f}}$, $\chi$ and $\textbf{\textit{h}}$ with
\begin{eqnarray*}
\textbf{\textit{f}}\in (\textbf{\textit{T}}^{2}_{k+1}(\Omega))',\quad \chi\in W^{0,2}_{-k}(\Omega),\quad \textbf{\textit{h}}\in H^{-3/2}(\Gamma),
\end{eqnarray*}
such that $\textbf{\textit{h}}\cdot\textbf{\textit{n}}=0$ on $\Gamma$ and for any $(\boldsymbol{\xi},\eta)\in \mathcal{N}_{k+1}^{2,2}(\Omega)$,\\
\begin{equation}\label{condition de compatibilite solution faible p=2}
\langle \textbf{\textit{f }},\boldsymbol{\xi}\rangle_{(\textbf{\textit{T}}^{2}_{k+1}(\Omega))'\times \textbf{\textit{T}}^{2}_{k+1}(\Omega)}-\displaystyle\int_{\Omega}
\chi\,{\eta}d\textbf{\textit{x}}+\langle\textbf{\textit{h}},
\boldsymbol{\xi}\rangle_{H^{-3/2}(\Gamma)\times H^{3/2}(\Gamma)}=0,
\end{equation}
then, problem $(\mathcal{S}_{T})$ has a solution $(\textbf{\textit{u}} ,\pi)\in W^{0,2}_{-k-1}(\Omega)\times W^{-1,2}_{-k-1}(\Omega)$ unique up to an element of $\mathcal{N}_{-k+1}^{2,2}(\Omega)$.
% Moreover, we have the estimate:
%\begin{eqnarray*}\label{estimate very weak solution}
%\inf _{(\boldsymbol{\lambda},\mu)\in\mathcal{N}^{0,2}_{-k-1}(\Omega)}\Big(\Vert\textbf{\textit{u}}+\boldsymbol{\lambda}\Vert_{W^{0,2}_{-k-1}(\Omega)}+\Vert\pi+\mu\Vert_{W^{-1,2}_{-k-1}(\Omega)}\Big)\leqslant C\Big(\Vert\textbf{\textit{f}}\Vert_{(\textbf{\textit{T}}_{k+1}(\Omega))'}+\Vert\chi\Vert_{W^{0,2}_{-k}(\Omega)}+\Vert\textbf{\textit{h}}\Vert_{H^{-3/2}(\Gamma)}\Big).
%\end{eqnarray*}
\end{theo}

\begin{proof}
The proof of the following theorem is similar to~\cite[Theorem 7]{LMR_2020}.
%Observe first that the uniqueness is a straightforward consequence of Proposition~\ref{caracterisation du noyau dans non borne}. 
Observe that, in view of Proposition~\ref{proposition equivalence p=2}, if the pair $(\textit{\textbf{u}},\pi)$ that belongs to $ W_{-k-1}^{0,2}(\Omega)\times W_{-k-1}^{-1,2}(\Omega)$ is a solution of ($\mathcal{S}_T$), then for any $(\boldsymbol{\varphi},q)\in {S}_{k+1}(\Omega)\times W^{1,2}_{k+1}(\Omega)$, adding~\eqref{formulation faible I p=2} and~\eqref{formulation faible II p=2}, we have 
\begin{eqnarray}\label{form variationnel p=2}
\nonumber &&\displaystyle\int_{\Omega} \textbf{\textit{u}}\cdot\big(- \Delta\,\boldsymbol{\varphi}+\nabla\,q\big) d\textbf{x}-\langle\pi,\mathrm{div}\,\boldsymbol{\varphi}\rangle_{W^{-1,2}_{-k-1}(\Omega)\times\mathring{ W}^{1,2}_{k+1}(\Omega)}
=\langle\textbf{\textit{f}},\boldsymbol{\varphi}\rangle_{(\textbf{\textit{T}}_{k+1}^{}(\Omega))'\times \textbf{\textit{T}}_{k+1}^{}(\Omega)}\\
 &&
+\langle\textbf{\textit{h}},\boldsymbol{\varphi}\rangle_{ H^{-3/2}(\Gamma)\times  H^{3/2}(\Gamma)}-\displaystyle\int_{\Omega} \chi q d\textbf{x}+\langle g,q-2[\textbf{D}(\boldsymbol{\varphi})\textbf{\textit{n}}]\cdot\textbf{\textit{n}}\rangle_{H^{-1/2}(\Gamma)\times H^{1/2}(\Gamma)}.
\end{eqnarray}
In particular, if $\forall (\boldsymbol{\varphi},q)\in \mathcal{N}_{k+1}^{2,2}(\Omega)$ we obtain~\eqref{condition de compatibilite solution faible p=2}.\\

\noindent It remains now to look for $(\textit{\textbf{u}},\pi)$ belongs to $ W_{-k-1}^{0,2}(\Omega)\times W_{-k-1}^{-1,2}(\Omega)$ satisfying~\eqref{form variationnel p=2}.
To that end, let $\textbf{T}$ be the linear form defined from $W^{0,2}_{k+1}(\Omega)\times \mathring{ W}^{1,2}_{k+1}(\Omega)\perp \mathcal{N}_{-k+1}^{2,2}(\Omega) $ onto $\R$ by
\begin{eqnarray*}
\textbf{T}(\textbf{\textit{F}}, \theta)=\langle\textbf{\textit{f}},\boldsymbol{\varphi}\rangle_{(T_{k+1}^{}(\Omega))'\times T_{k+1}^{}(\Omega)}+\langle\textbf{\textit{h}},\boldsymbol{\varphi}\rangle_{H^{-3/2}(\Gamma)\times H^{3/2}(\Gamma)}-\displaystyle\int_{\Omega} \chi q d\textbf{x},
\end{eqnarray*}
where the pair $(\boldsymbol{\varphi},q)\in W^{2,2}_{k+1}(\Omega)\times W^{1,2}_{k+1}(\Omega)$ is a solution of the following problem (see Theorem~\ref{solution forte}):
\begin{eqnarray*}
\begin{cases}
-\Delta\,\boldsymbol{\varphi}+\nabla\,{q}=\textbf{\textit{F}}\quad and\quad
\mathrm{div}\,\boldsymbol{\varphi}=-{\theta}\qquad\quad\text{ in }\,\Omega,\\
\quad\boldsymbol{\varphi}\cdot \textbf{\textit{n}}=0\,\,\,\,and\quad
2[\mathrm{\textbf{D}}(\boldsymbol{\varphi})\textbf{\textit{n}}]_{\tau}+\alpha\boldsymbol{\varphi}_{\tau}=\textbf{\textit{0}}\quad\text{ on }\,\Gamma

\end{cases}
\end{eqnarray*}
and satisfying the following estimate:
\small{\begin{eqnarray}\label{estimation solution réguliere p=2}
\inf _{(\boldsymbol{\lambda},\mu)\in\mathcal{N}_{-k+1}^{2,2}(\Omega)}\Big(\Vert \boldsymbol{\varphi}+\boldsymbol{\lambda}\Vert_{\textbf{\textit{W}}^{2,2}_{k+1}(\Omega)}+\Vert {q}+\mu \Vert_{{{W}}^{1,2}_{k+1}(\Omega)}\Big)\leqslant C\Big(\Vert \textbf{\textit{F}}\Vert_{W^{0,2}_{k+1}(\Omega)}+\Vert\theta\Vert_{W^{1,2}_{k+1}(\Omega)}\Big).
\end{eqnarray}}
Then for any pair $(\textbf{\textit{F}},\theta)\in W^{0,2}_{k+1}(\Omega)\times W^{1,2}_{k+1}(\Omega)$ and for any $(\boldsymbol{\lambda},\mu)\in\mathcal{N}_{-k+1}^{2,2}(\Omega)$, we can write
\begin{eqnarray*}
\vert\textbf{T}(\textbf{\textit{F}}, \theta)\vert &= &\Big\vert\langle\textbf{\textit{f}},\boldsymbol{\varphi}\rangle_{(T_{k+1}^{}(\Omega))'\times T_{k+1}^{}(\Omega)}+\langle\textbf{\textit{h}},\boldsymbol{\varphi}\rangle_{H^{-3/2}(\Gamma)\times H^{3/2}(\Gamma)}-\displaystyle\int_{\Omega} \chi q d\textbf{x}\Big\vert\\
&=&\Big\vert\langle\textbf{\textit{f}},\boldsymbol{\varphi}+\boldsymbol{\lambda}\rangle_{(T_{k+1}^{}(\Omega))'\times T_{k+1}^{}(\Omega)}+\langle\textbf{\textit{h}},\boldsymbol{\varphi}+\boldsymbol{\lambda}\rangle_{H^{-3/2}(\Gamma)\times H^{3/2}(\Gamma)}-\displaystyle\int_{\Omega} \chi (q+\mu) d\textbf{x}\Big\vert\\
&\leqslant & C\Big(\Vert\textbf{\textit{f}}\Vert_{(\textbf{\textit{T}}^{}_{k+1}(\Omega))'}+\Vert\chi\Vert_{W^{0,2}_{-k}(\Omega)}+\Vert\textbf{\textit{h}}\Vert_{H^{-3/2}(\Gamma)}\Big)\Big(\Vert \boldsymbol{\varphi}+\boldsymbol{\lambda}\Vert_{\textbf{\textit{W}}^{2,2}_{k+1}(\Omega)}+\Vert {q}+\mu \Vert_{{{W}}^{1,2}_{k+1}(\Omega)}\Big).
\end{eqnarray*}
Using~\eqref{estimation solution réguliere p=2}, we have
\begin{eqnarray*}
\vert\textbf{T}(\textbf{\textit{F}}, \theta)\vert &\leqslant & C\inf _{(\boldsymbol{\lambda},\mu)\in\mathcal{N}_{-k+1}^{2,2}(\Omega)}\Big(\Vert \boldsymbol{\varphi}+\boldsymbol{\lambda}\Vert_{\textbf{\textit{W}}^{2,2}_{k+1}(\Omega)}+\Vert {q}+\mu \Vert_{{{W}}^{1,2}_{k+1}(\Omega)}\Big)\\
&\leqslant &  C\Big(\Vert \textbf{\textit{F}}\Vert_{W^{0,2}_{k+1}(\Omega)}+\Vert\theta\Vert_{W^{1,2}_{k+1}(\Omega)}\Big)
\end{eqnarray*}

\noindent From this, we can deduce that the linear form $\textbf{T}$ is continuous on the space $W^{0,2}_{k+1}(\Omega)\times\mathring{ W}^{1,2}_{k+1}(\Omega)\perp \mathcal{N}_{-k+1}^{2,2}(\Omega)$. Since the dual space of $W^{0,2}_{k+1}(\Omega)\times \mathring{ W}^{1,2}_{k+1}(\Omega)\perp \mathcal{N}_{-k+1}^{2,2}(\Omega)$ is $W^{0,2}_{-k-1}(\Omega)\times W^{-1,2}_{-k-1}(\Omega)/ \mathcal{N}_{-k+1}^{2,2}(\Omega)$. From Riesz's Representation theorem, we deduce that there exists a unique  $(\textbf{\textit{u}},\pi)$ belongs to $W^{0,2}_{-k-1}(\Omega)\times W^{-1,2}_{-k-1}(\Omega)/\mathcal{N}_{-k+1}^{2,2}(\Omega)$ satisfying ~\eqref{form variationnel p=2}.
\end{proof}
\noindent Now, we will finish with the case that $g$ is not vanish. For this reason we need the following lemma concerning the existence and uniqueness for the exterior Neumann problem.
\begin{lemm}\label{theorem Neumann generalized}
For any $f$ in $L^{2}(\Omega)$ and $g$ in $H^{-1/2}(\Gamma)$. Then, the problem :
\begin{eqnarray}\label{Neumann problem}
-\Delta\,u=f\quad\text{ in }\quad\Omega,\quad\dfrac{\partial u}{\partial\textbf{\textit{n}}}=g\quad\text{ on }
\quad\Gamma,
\end{eqnarray}
 has a unique solution $u\in W^{1,2}_{-1}(\Omega)$ and we have the following estimate: 
\begin{eqnarray}\label{estimation faible neumann}
\Vert u\Vert_{W^{1,2}_{-1}(\Omega)} \leqslant C\big(\Vert f\Vert_{L^{2}(\Omega)}+\Vert g\Vert_{H^{-1/2}(\Omega)}\big)
\end{eqnarray}
\end{lemm}
%%%%%%%%%%%%%%%%%%%%%%%%%%%%%%%%%%%%%%
\begin{proof}
Let us extend $f$ by zero in $\Omega'$ and let $\widetilde{f}$ denote the extended function. Then $\widetilde{f}$ belongs to $L^{2}(\R^3)$. Thanks \cite[Theorem 3.9]{Amrouche_1994}, there exists a unique function $\widetilde{v} \in W^{2,2}_{0}(\R^3)$ such that
\begin{eqnarray}\label{laplace in R^3}
-\Delta\,\widetilde{v}=\widetilde{f}\quad\text{ in }\,\,\R^3.
\end{eqnarray}
Then $\nabla\,\widetilde{v}\cdot n$ belongs to $H^{1/2}(\Gamma)\hookrightarrow H^{-1/2}(\Gamma)$. It follows from \cite[Theorem 2.7]{Louati_Meslameni_Razafison}, that the following problem:
\begin{eqnarray}\label{Neumann Harmonic}
\Delta\,w=0\quad\text{ in }\,\,\,\Omega\quad\text{ and } \quad\nabla\,w\cdot n=g-\nabla\,\widetilde{v}\cdot n\quad\text{ on }\,\,\,\Gamma,
\end{eqnarray}
has a unique solution $w\in W^{1,2}_{-1}(\Omega)$. Thus $u=\widetilde{v}_{\mid_{\Omega}}
+w\in W^{1,2}_{-1}(\Omega)$ is the required solution of~\eqref{Neumann problem}. 
\end{proof}
\begin{theo}\label{g neq 0}

Given any $\textbf{\textit{f}}$, $\chi$, $g$ and $\textbf{\textit{h}}$ with
\begin{eqnarray*}
\textbf{\textit{f}}\in (\textbf{\textit{T}}^{2}_{k+1}(\Omega))',\quad \chi\in W^{0,2}_{-k}(\Omega),\quad g\in H^{-1/2}(\Gamma),\quad \textbf{\textit{h}}\in H^{-3/2}(\Gamma),
\end{eqnarray*}
such that $\textbf{\textit{h}}\cdot\textbf{\textit{n}}=0$ on $\Gamma$ and for any $(\boldsymbol{\xi},\eta)\in \mathcal{N}_{k+1}^{2,2}(\Omega)$,\\
\begin{equation}\label{condition de compatibilite solution faible p=22}
\langle \textbf{\textit{f }},\boldsymbol{\xi}\rangle_{(\textbf{\textit{T}}^{2}_{k+1}(\Omega))'\times \textbf{\textit{T}}^{2}_{k+1}(\Omega)}-\displaystyle\int_{\Omega}
\chi\,{\eta}d\textbf{\textit{x}}=\langle g,2[\textrm{\textbf{D}}(\boldsymbol{\xi})\textbf{\textit{n}}]\cdot\textbf{\textit{n}}-\eta \rangle_{H^{-1/2}(\Gamma)\times H^{1/2}(\Gamma)}-\langle\textbf{\textit{h}},
\boldsymbol{\xi}\rangle_{H^{-3/2}(\Gamma)\times H^{3/2}(\Gamma)},
\end{equation}
then, problem $(\mathcal{S}_{T})$ has a solution $(\textbf{\textit{u}} ,\pi)\in W^{0,2}_{-k-1}(\Omega)\times W^{-1,2}_{-k-1}(\Omega)$ unique up to an element of $\mathcal{N}_{-k+1}^{2,2}(\Omega)$.
% Moreover, we have the estimate:
%\begin{eqnarray*}\label{estimate very weak solution}
%\inf _{(\boldsymbol{\lambda},\mu)\in\mathcal{N}_{-k-1}^{0,2}(\Omega)}\Big(\Vert\textbf{\textit{u}}+\boldsymbol{\lambda}\Vert_{W^{0,2}_{-k-1}(\Omega)}+\Vert\pi +\mu \Vert_{W^{-1,2}_{-k-1}(\Omega)}\Big)\leqslant C\Big(\Vert\textbf{\textit{f}}\Vert_{(\textbf{\textit{T}}_{k+1}(\Omega))'}+\Vert\chi\Vert_{W^{0,2}_{-k}(\Omega)}+\Vert g\Vert_{ H^{-1/2}(\Gamma)}+\Vert\textbf{\textit{h}}\Vert_{H^{-3/2}(\Gamma)}\Big).
%\end{eqnarray*}
\end{theo}
\begin{proof}

Let $g\in H^{-1/2}(\Gamma)$. Consider the following exterior Neumann problem
% According Theorem~\ref{theorem Neumann}, then there exists $\Theta\in W^{1,p}_{-1}(\Omega)$ solution of the following problem:
\begin{eqnarray}\label{problem Neumann p=2}
\Delta\,v=0\quad\text{ in } \Omega\quad\text{ and }\quad \nabla\,v\,\cdot\textbf{\textit{n}}=g\quad\text{ on }\Gamma.
\end{eqnarray}
\noindent Then thanks to Lemma~\ref{theorem Neumann generalized}, this problem has a solution $v$ that belongs to $W^{1,2}_{-1}(\Omega)$. Let $R$ be a positive real number large enough so that $\overline{\Omega'}\subset B_{R}$ and let $\psi \in \mathcal{D}(\overline{\Omega})$ such that $0\le\psi\le 1$,
supp~$\psi~\subset~\overline{\Omega}_{R+1}$ and $\psi=1$ in $\overline{\Omega}_R$. 
Set now $w=v\psi,$ then $w$ has obviously a compact support and thus $w$ belongs to $W^{1,2}_{-k-1}(\Omega)$. Next, $\Delta\,w=v\,\Delta\,\psi+2\nabla\,v\,\nabla\,\psi$ has a compact support and thus belongs to $W^{0,2}_{-k}(\Omega)$.\\

\noindent Next we look for a pair $(\textbf{\textit{z}},{\pi})\in W^{0,2}_{-k-1}(\Omega)\times {{W}}^{-1,2}_{-k-1}(\Omega)$ satisfying
\begin{eqnarray}\label{problem relevement g p=2}
\begin{cases}
-\Delta\,\textbf{\textit{z}}+\nabla\,{\pi}=\textbf{\textit{f}}+\Delta(\nabla\,{w})\quad\text{and}\quad
\mathrm{div}\,\textbf{\textit{z}}=\chi-\Delta\,w\quad\text{in}\quad \Omega,\\[6pt]
\textbf{\textit{z}}\cdot \textbf{\textit{n}}=0\quad\text{and}\quad
2[\mathrm{\textbf{D}}(\textbf{\textit{z}})\textbf{\textit{n}}]_{\tau}+\alpha\textbf{\textit{z}}_{\tau}=\textbf{\textit{h'}}\quad\text{on}\quad\Gamma,
\end{cases}
\end{eqnarray}
where $\textbf{\textit{h'}}=\textbf{\textit{h}}-2[\mathrm{\textbf{D}}(\nabla\,{w})\textbf{\textit{n}}]_{\tau}-\alpha (\nabla\,w)_{\tau}$.

\noindent Observe that in view of the assumption on $\textit{\textbf{f }}$ and Lemma~\ref{caractérisation T^2},
$\textbf{\textit{f}}+\Delta\,(\nabla\,{w})$ belongs to $(T^{2}_{k+1}(\Omega))'$. Besides, as $w$ belongs to $W^{1,2}_{-k-1}(\Omega)$, 
then\, $\nabla\,w$ belongs to $W^{0,2}_{-k-1}(\Omega)$ and $\Delta(\nabla\,w)$ belongs to $(T^{2}_{k+1}(\Omega))'$. As a consequence,\, $\textbf{\textit{h'}}$\, belongs to \,$H^{-3/2}(\Gamma)$\,
%and satisfies 
%\begin{equation}
%\label{esti.h.theta}
%\|\textit{\textbf{h' }}\|_{H^{-1/2}(\Gamma)}\le C\left(\|\textit{\textbf{h }}\|_{H^{-1/2}(\Gamma)}+\|\theta\|_{W_0^2(\Omega)/\R}\right)
%\end{equation}
 and clearly satisfies\, $\textbf{\textit{h'}}\cdot\textbf{\textit{n}}=0$ \, on $\Gamma$. Thus, due to the Theorem~\ref{solutions très faibles p=2}, problem \eqref{problem relevement g p=2} has a solution $(\textbf{\textit{z}},{\pi})\in W^{0,2}_{-k-1}(\Omega)\times {{W}}^{-1,2}_{-k-1}(\Omega)$\, if the following condition is satisfied.
\begin{eqnarray}\label{CC lié très faible p=2}
\nonumber\forall (\boldsymbol{\xi},\eta)\in \mathcal{N}_{k+1}^{2,2}(\Omega),&&\quad
\langle\textbf{\textit{f }}+\Delta\,(\nabla w),\boldsymbol{\xi}\rangle_{(T^{2}_{k+1}(\Omega))'\times T^{2}_{k+1}(\Omega)}-\displaystyle\int_{\Omega}
(\chi-\Delta\,w)\,{\eta}d\textbf{\textit{x}}\\
 &&+\langle\textbf{\textit{h}}',\boldsymbol{\xi}\rangle_{H^{-3/2}(\Gamma)\times H^{3/2}(\Gamma)}=0.
\end{eqnarray}
Let $(\boldsymbol{\xi},\eta)$ be in $\mathcal{N}_{k+1}^{2,2}(\Omega)$. For any ($\boldsymbol{\varphi},\psi)\in\mathcal{D}(\overline{\Omega})\times\mathcal{D}(\overline{\Omega}),$ we have
\begin{eqnarray}\label{Formule de green pour cc p=2}
 &&\displaystyle\int_{\Omega}\big[\big(-\Delta\,\boldsymbol{\varphi}+\nabla\,\psi\big)\cdot\boldsymbol{\xi}-\eta\,\mathrm{div}\,\boldsymbol{\varphi}\big] d\textbf{\textit{x}}\\
\nonumber &&=
\langle{\boldsymbol{\varphi}\cdot\textbf{\textit{n}}},2[\mathrm{\textbf{D}}(\boldsymbol{\xi})\textbf{\textit{n}}]\cdot{\textbf{\textit{n}}}-\eta \rangle_{H^{-1/2}(\Gamma)\times H^{1/2}(\Gamma)}
-\langle 2[\mathrm{\textbf{D}}(\boldsymbol{\varphi})\textbf{\textit{n}}]_{\tau}+\alpha\,\boldsymbol{\varphi}_{\tau},{\xi}_{\tau}\rangle_{H^{-3/2}(\Gamma)\times H^{3/2}(\Gamma)}.
\end{eqnarray}
In particular, the Green's formula~\eqref{Formule de green pour cc p=2} holds for any  $\boldsymbol{\varphi}=\nabla\,w\in W^{0,2}_{-k-1}(\Omega)$ and $\psi=0$. We obtain
\begin{eqnarray}\label{formule de green pour nabla w p=2}
 && \langle -\Delta\,(\nabla\,w),\boldsymbol{\xi} \rangle_{(\textbf{\textit{T}}^{2}_{k+1}(\Omega))'\times \textbf{\textit{T}}^{2}_{k+1}(\Omega)}-\displaystyle\int_{\Omega}\Delta \,w\,\,\eta d\textbf{\textit{x}}\\
\nonumber &&=
\langle g, 2[\mathrm{\textbf{D}}(\boldsymbol{\xi})\textbf{\textit{n}}]\cdot{\textbf{\textit{n}}}-\eta \rangle_{H^{-1/2}(\Gamma)\times H^{1/2}(\Gamma)}
- \langle \textbf{\textit{h}}-\textbf{\textit{h}}', \boldsymbol{\xi}\rangle_{H^{-3/2}(\Gamma)\times H^{3/2}(\Gamma)}=0.
\end{eqnarray}

\noindent Combining \eqref{condition de compatibilite solution faible p=22} and \eqref{formule de green pour nabla w p=2} allows to obtain~\eqref{CC lié très faible p=2}.
Thus setting $\textbf{\textit{u}}=\textit{\textbf{z}}+\nabla\,{w}$, the pair $(\textbf{\textit{u}}, {\pi}) \in  W^{0,2}_{-k-1}(\Omega) \times {{W}}^{-1,2}_{-k-1} (\Omega)$ 
is the solution of $(\mathcal{S}_{T})$.

\end{proof}
\section*{Conclusion}
In this paper, we solved the exterior Stokes problem with the Navier slip boundary conditions with a positive friction term. To prescribe the growths or decay of functions at infinity, we set the problem in weighted Sobolev spaces. Our study is based on an $L^2$-theory. We established the existence and the uniqueness of the very weak solution to the problem. The main idea consists in the use of a duality argument with the strong solutions obtained in~\cite{DMR-2019}. In a forthcoming work, we will study the very weak solution to the problem in weighted $L^p$ spaces.
%%%%%%%%%%%%%%%%%

\end{document}